\documentclass[12pt,a4paper]{article}
\usepackage[utf8]{inputenc}
\usepackage{amsmath}
\usepackage{amsfonts}
\usepackage{amssymb}
\usepackage[left=2cm,right=2cm,top=2cm,bottom=2cm]{geometry}
\usepackage{amsfonts,amsthm}
\usepackage{graphicx}
\usepackage{epstopdf}
\usepackage{algorithm}
\usepackage{algorithmic}
\usepackage{amsmath}
\usepackage{dsfont}
\usepackage{amssymb}
\usepackage{bm}
\usepackage{makecell}
\usepackage{float}
\usepackage{mathtools}
\usepackage{psfrag}
\usepackage{multicol}
\usepackage{subcaption}
\usepackage{colortbl}
\usepackage{pgfplots}
\usepackage{tikz}
\newcommand\red[1]{\textcolor{red}{#1}}

\newtheorem{theorem}{Theorem}[section]
\newtheorem{lemma}[theorem]{Lemma}
\newtheorem{corollary}[theorem]{Corollary}

\newtheorem{remark}[theorem]{Remark}
\usetikzlibrary{external}
\ifpdf
\DeclareGraphicsExtensions{.eps,.pdf,.png,.jpg}
\else
\DeclareGraphicsExtensions{.eps}
\fi

\usepackage{hyperref}
\newcommand{\footremember}[2]{%
	\footnote{#2}
	\newcounter{#1}
	\setcounter{#1}{\value{footnote}}%
}
\newcommand{\footrecall}[1]{%
	\footnotemark[\value{#1}]%
} 

\author{%
	C\'ecile Haberstich\footremember{alley}{CEA, DAM, DIF, F-91297 Arpajon France}\footremember{trailer}{Centrale Nantes, LMJL UMR CNRS 6629, France}%
	\and Anthony Nouy\footrecall{trailer} %
	\and Guillaume Perrin\footrecall{alley} %
}
\usepackage{cleveref}
\title{Boosted optimal weighted least-squares}
\date{}

\begin{document}
	
	\maketitle
	
	% REQUIRED
	\begin{abstract}
This paper is concerned with the approximation of a function $u$ in a given approximation space $V_m$ of dimension $m$ from evaluations of the function at $n$ suitably chosen points. The aim is to construct an approximation of $u$ in $V_m$ which yields an error close to the best approximation error in $V_m$ and using as few evaluations as possible. Classical least-squares regression, which defines a projection in $V_m$ from $n$ random points, usually requires a large $n$ to guarantee a stable approximation and an error close to the best approximation error. This is a major drawback for applications where $u$ is expensive to evaluate. 
One remedy is to use a weighted least-squares projection using $n$ samples drawn from a properly selected distribution.
In this paper, we introduce a boosted weighted least-squares method which allows to ensure almost surely the stability of the weighted least-squares projection with a sample size close to the interpolation regime $n=m$. It consists in sampling according to a measure associated with the optimization of a stability criterion over a collection of independent $n$-samples, and resampling according to this measure until a stability condition is satisfied. A greedy method is then proposed to remove points from the obtained sample. Quasi-optimality properties in expectation are obtained for the weighted least-squares projection, with or without the greedy procedure. The proposed method is validated on numerical examples and compared to state-of-the-art interpolation and weighted least-squares methods.
	\end{abstract}
	
	% REQUIRED
	
	% REQUIRED
	%\begin{AMS}
	%  41A10, 41A65, 93E24, 65D05, 65D15
	%\end{AMS}
	\section{Introduction}
	~
	The continuous improvement of computational resources makes the role of the numerical simulation always more important for modelling complex systems. However most of these numerical simulations remain very costly from a computational point of view. Furthermore, for many problems such as optimization, estimation or uncertainty quantification, the model is a function of possibly numerous parameters (design variables, uncertain parameters...) and has to be evaluated for many instances of the parameters.
	One remedy is to build an approximation of this function of the parameters which is further used as a surrogate model, or as a companion model used as a low-fidelity model. 
	
	This paper is concerned with the approximation of a function $u$ using evaluations of the function at suitably chosen points. We consider functions from $L^2_\mu(\mathcal {X})$, the space of square-integrable functions defined on a set $\mathcal X$ equipped with a probability measure $\mu$. Given an approximation space $V_m$ of dimension $m$ in $L^2_\mu(\mathcal X)$,  
	the aim is to construct an approximation of $u$ in $V_m$ which yields an error close to the best approximation error in $V_m$ and using as few evaluations as possible. 
	A classical approach is least-squares regression, which defines the approximation by solving  
	$$
	\min_{v\in V_m} \frac{1}{n} \sum_{i=1}^{n} (u(x^i) - v(x^i))^2,
	$$
	where the $x^i$ are i.i.d. samples drawn from the measure $\mu$.
	However, to guarantee a stable approximation and an error close to the best approximation error, least-squares regression may require a sample size $n$ much higher than $m$ (see \cite{Cohen2013}). This issue can be overcome by weighted least-squares projection, which is obtained by solving
	$$
	\min_{v\in V_m} \frac{1}{n} \sum_{i=1}^{n} w(x^i) (u(x^i) - v(x^i))^2,
	$$
	where the $x^i$ are points not necessarily drawn from $\mu$ and the $w(x^i)$ are corresponding weights. A suitable choice of weights and points may allow to decrease the sample size to reach the same approximation error, see e.g.  \cite{Doostan2015,narayan2017christoffel}. 
	In \cite{Cohen2016}, the authors introduce an optimal sampling measure $\rho$ with a density  $w(x)^{-1}$ with respect to the reference measure $\mu$ which depends on the approximation space. Choosing i.i.d. samples $x^i$ from this optimal measure, one obtains with high probability $1-\eta$ a stable approximation and an error of the order of the best approximation error using a sample size $n$ in $\mathcal{O}(m \log(m \eta^{-1}))$. Nevertheless, the necessary condition for having stability requires a sample size $n$ much higher than $m$, especially when a small probability $\eta$ is desired.\\
	
	Here we introduce 
	a boosted least-squares method which enables us to ensure almost surely the stability of the weighted least-squares projection in expectation with a sample size close to the interpolation regime $n=m$. It consists in sampling according to a measure associated with the optimization of a stability criterion over a collection of independent $n$-samples, and resampling according to this measure until a stability condition is satisfied. A greedy method is then proposed to remove points from the obtained sample. Quasi-optimality properties in expectation are obtained for the weighted least-squares projection, with or without the greedy procedure.\\
	
	If the observations are polluted by a noise, here modeled by a random variable $e$, then the weighted least-squares projection is defined as the solution of 
	$$
	\min_{v\in V_m} \frac{1}{n} \sum_{i=1}^{n} (y^i - v(x^i))^2,
	$$ 
	where $y^i = u(x^i)+ e^i$, with $\{e^i\}_{i=1}^n$ i.i.d realizations of the random variable $e$. Quasi-optimality property is lost in the case of noisy observations, because of an additional error term due to the noise. This latter error term can however be reduced by increasing $n$.\\
	
	The outline of the paper is as follows. In Section \ref{sec:weighted-least-squares}, we introduce the theoretical framework, and present some useful results on weighted least-squares projections. We recall the optimal sampling measure from \cite{Cohen2016}, and outline its limitations. In Section \ref{sec:boosted}, we present the boosted least-squares approach and analyze it in the noise-free case. The theoretical results are extended to the noisy case in Section \ref{sec:noisy_case}. In Section \ref{sec:examples}, we present numerical examples.         
	
	\section{Least-squares method}\label{sec:weighted-least-squares}
	Let $\mathcal{X} $ be a subset of $ \mathbb{R}^d$ equipped with a probability measure $\mu$, with $d\ge 1$. We consider a function $u$ from  $L^2_{\mu}(\mathcal{X})$, the Hilbert space of square-integrable real-valued functions defined on $\mathcal{X}$.  We let $\Vert \cdot \Vert_{L^2_{\mu}}$ be the natural norm in $L^2_{\mu}(\mathcal{X})$ defined by
	\begin{equation}
	\Vert v \Vert^2_{L^2_{\mu}} = \int_{\mathcal{X}}  v(x) ^2 d\mu(x).
	\end{equation}
	When there is no ambiguity, $L^2_{\mu}(\mathcal{X})$ will be simply denoted $L^2_{\mu}$, and the norm $\Vert v \Vert^2_{L^2_{\mu}}$  and associated inner product $(\cdot,\cdot)_{L^2_\mu}$ will be denoted $\Vert \cdot \Vert$ and $(\cdot,\cdot)$ respectively.  Let $V_m$ be a $m$-dimensional subspace of $L^2_{\mu}$, with $m\ge 1$, and $\{\varphi_j\}_{j=1}^m$ be an orthonormal basis of $V_m$. The best approximation of $u$ in $V_m$ is given by its orthogonal projection defined by
	\begin{equation}
	P_{V_m}u := \arg\min_{v \in V_m} \Vert u - v \Vert.
	\end{equation}
	
	\subsection{Weighted least-squares projection}
	\label{sec:wls}
	Letting  $\bm{x}^n:=\{x^i\}_{i=1}^n$ be a set of $n$ points in $\mathcal{X}$, we consider 
	the weighted least-squares projection defined by
	\begin{equation}
	Q_{V_m}^{\bm x^n}u := \arg\min_{v \in {V_m}} \Vert u - v \Vert_{\bm{x}^n}, \label{wls}
	\end{equation}
	where $\Vert \cdot \Vert_{\bm{x}^n}$ is a discrete semi-norm defined for $v$ in $L^2_\mu$ by
	\begin{equation}
	\Vert v \Vert_{\bm{x}^n}^2 := \frac{1}{n} \sum_{i=1}^n w(x^i) v(x^i) ^2, 
	\end{equation}
	where $w$ is a given non negative function defined on $\mathcal{X}$. 
	We denote by $$\bm{\varphi} = (\varphi_1, \hdots, \varphi_m): \mathcal{X} \to \mathbb{R}^m$$  the $m$-dimensional vector-valued function such that $\bm{\varphi}(x) =  (\varphi_1(x), \hdots, \varphi_m(x))^T$, 
	%$\bm{\Phi} = (\varphi_j(x^i))_{1 \le i \le n, 1 \le j \le m}$ the matrix of evaluations of $\bm{\varphi}$ at points $\bm{x} ^n$, 
	and by $\bm{G}_{\bm{x}^n}$ the  empirical Gram matrix defined by 
	\begin{equation}
	\bm{G}_{\bm{x}^n} := \frac{1}{n}\sum_{i=1}^n w(x^i)\bm{\varphi}(x^i)\otimes \bm{\varphi}(x^i).
	\end{equation}
	The stability of the weighted least-squares projection can be characterized by  
	$$Z_{\bm{x}^n} :=  \Vert \bm{G}_{\bm{x}^n}- \bm{I} \Vert_2,$$ 
	which measures 
	a distance between the empirical Gram matrix and the identity matrix $\bm{I}$, with $\Vert \cdot \Vert_2$ the matrix spectral norm. For any $v$ in $V_m$, we have 
	\begin{equation}
	(1-Z_{\bm{x}^n})  \Vert v \Vert^2 \le \Vert v \Vert_{\bm{x}^n}^2  \le (1+Z_{\bm{x}^n}) \Vert v \Vert^2.\label{norm-equiv-Z}
	\end{equation} 
	We have the following properties that will be useful in subsequent analyses. 
	\begin{lemma}\label{lem:quasi-optim-under-condition}
		Let $\bm x^n$ be a set of $n$ points in $\mathcal{X}$ such that $Z_{\bm{x}^n} =  \Vert \bm{G}_{\bm{x}^n}- \bm{I} \Vert_2\le \delta$ for some $\delta \in [0,1)$. Then 
		\begin{equation}
		(1-\delta)  \Vert v \Vert^2 \le \Vert v \Vert_{\bm{x}^n}^2  \le (1+ \delta) \Vert v \Vert^2\label{norm-equiv-delta}
		\end{equation}
		and the weighted least-squares projection $Q_{V_m}^{\bm x^n} u$ associated with $\bm x^n$ satisfies 
		\begin{equation}
		\Vert u - Q_{V_m}^{\bm x^n} u \Vert^2 \le   \Vert u - P_{V_m}u \Vert^2 + (1-\delta)^{-1} \Vert u - P_{V_m}u \Vert_{\bm x^n}^2.\label{quasi-optim-L2-under-condition}
		\end{equation}
		%and
		%\begin{equation}
		%\Vert u - Q_{V_m}^{\bm x^n} u \Vert \le  \Vert u - v \Vert + (1-\delta)^{-1/2} \Vert u - v \Vert_{\bm x^n}
		%\label{link-L2-Linf-under-condition}
		%\end{equation}
		%for all $v\in V_m$.
	\end{lemma}
	\begin{proof}
		The property \eqref{norm-equiv-delta} directly follows from \eqref{norm-equiv-Z} and $Z_{\bm x^n}\le \delta$.
		Using the property of the orthogonal  projection $P_{V_m}u$ and \eqref{norm-equiv-delta}, we have that 
		\begin{equation*}
		\begin{aligned}
		\Vert u - Q_{V_m}^{\bm x^n} u \Vert^2 & = \Vert u - P_{V_m} u \Vert^2 + \Vert  P_{V_m} u - Q_{V_m}^{\bm x^n} u  \Vert^2\\
		&   \le \Vert u - P_{V_m} u \Vert^2 + (1-\delta)^{-1}\Vert  P_{V_m} u - Q_{V_m}^{\bm x^n} u  \Vert_{\bm x^n}^2.
		\end{aligned}
		\end{equation*}
		Using the fact that $Q_{V_m}^{\bm x^n}$ is an orthogonal projection on $V_m$ with respect to the semi-norm $\Vert \cdot \Vert_{\bm x^n}$, we have that for any $v$, $\Vert Q_{V_m}^{\bm x^n} v \Vert_{\bm x^n} \le \Vert v \Vert_{\bm x^n}$. We deduce that 
		$$
		\Vert  P_{V_m} u - Q_{V_m}^{\bm x^n} u  \Vert_{\bm x^n} = \Vert  Q_{V_m}^{\bm x^n} (P_{V_m} u - u)  \Vert_{\bm x^n} \le  \Vert   P_{V_m} u - u  \Vert_{\bm x^n},
		$$
		from which we deduce  \eqref{quasi-optim-L2-under-condition}. 
	\end{proof}
	We now provide a result which bounds the $L^2$ error by a best approximation error with respect to a weighted supremum norm.
	\begin{theorem}\label{th:L2Linf-conditioned}
		Let $\bm x^n$ be a set of $n$ points in $\mathcal{X}$ such that $Z_{\bm{x}^n} =  \Vert \bm{G}_{\bm{x}^n}- \bm{I} \Vert_2\le \delta$ for some $\delta \in [0,1)$. Then, 
		\begin{equation}
		\Vert u - Q_{V_m}^{\bm x^n} u \Vert \le  (B + (1-\delta)^{-1/2}) \inf_{v\in V_m} \Vert u - v \Vert_{\infty,w}\label{quasi-optim-L2-Linf-under-condition}
		\end{equation} 
		where $ B^2  = \int_{\mathcal{X}} w(x)^{-1}d\mu(x)$ and $\Vert v \Vert_{\infty,w} = \sup_{x \in \mathcal{X}} w(x)^{1/2} \vert v(x) \vert. $
	\end{theorem}	
	\begin{proof}
		Using Lemma \ref{lem:quasi-optim-under-condition} we note that for any $v\in V_m$,
		$$
		\Vert u - Q_{V_m}^{\bm x^n} u \Vert \le \Vert u - v \Vert + (1-\delta)^{-1/2} \Vert v - Q_{V_m}^{\bm x^n} u \Vert_{\bm x^n}, 
		$$
		and $\Vert v - Q_{V_m}^{\bm x^n} u \Vert_{\bm x^n}  =   \Vert Q_{V_m}^{\bm x^n}(v -  u) \Vert_{\bm x^n} \le \Vert u - v \Vert_{\bm x^n}$.\\ We then conclude by using the inequalities  $\Vert u-v \Vert_{\bm x^n} \le \Vert u- v \Vert_{\infty,w}$ and $\Vert u-v \Vert \le \left(\int_{\mathcal{X}} w(x)^{-1} d\mu(x) \right)^{1/2}\sup_{x\in \mathcal{X}} w(x)^{1/2} \vert u(x) -v(x) \vert. $
	\end{proof}
	In the case where $w^{-1}$ is the density of a probability measure with respect to $\mu$, (which will be the case in the rest of the paper), the constant $B$ from Theorem \ref{th:L2Linf-conditioned} is equal to 1. 
	\subsection{Random sampling}
	We consider the measure $\rho$ on $\mathcal{X}$ with density $w^{-1}$ with respect to $\mu$, i.e. $d\rho = w^{-1} d\mu$. If the $x^1,\hdots,x^n$ are i.i.d. random variables drawn from the measure $\rho$, or equivalently if $\bm{x}^n = (x^1,\hdots,x^n)$ is drawn from the product measure  $\rho^{\otimes n} := \bm \rho^n$ on $\mathcal{X}^n$, then for any function $v$ in $L^2_\mu$ (not only those in $V_m$), we have 
	\begin{equation}
	\mathbb{E}( \Vert v \Vert_{\bm{x}^n}^2 ) %=   \int v(x)^2 w(x) d\rho(x) =   \int v(x)^2  d\mu(x) 
	= \Vert v \Vert^2. \label{mean-discrete-norm}
	\end{equation}
	The condition \eqref{mean-discrete-norm} restricted to all functions $v\in V_m$ implies that the empirical Gram matrix 
	$\bm{G}_{\bm{x}^n}$ satisfies 
	\begin{equation}
	\mathbb{E}( \bm{G}_{\bm{x}^n}) =\frac{1}{n}\sum_{i=1}^n\mathbb{E}( w(x^i)\bm{\varphi}(x^i)\otimes \bm{\varphi}(x^i))=\bm{I}.
	\end{equation}
	The random variable $Z_{\bm{x}^n} = \Vert \bm{G}_{\bm{x}^n} - \bm I \Vert_2$ quantifies how much the random matrix $\bm{G}_{\bm{x}^n}$  deviates from its expectation. For any $\delta \in [0,1)$, if
	\begin{equation}
	\mathbb{P}(Z_{\bm{x}^n}  > \delta) \le \eta, \label{bound-PZ}
	\end{equation}
	then for all $v \in V_m$, Eq. \eqref{norm-equiv-delta} holds 
	with probability higher than $1-\eta.$ We directly conclude from Theorem \ref{th:L2Linf-conditioned} that 
	the weighted least-squares projection $Q_{V_m}^{{\bm{x}^n}}$ satisfies \eqref{quasi-optim-L2-Linf-under-condition}
	with probability higher than $1-\eta$ (and $B=1$). \\
	
	Now, we present results in expectation which relate the $L^2$-error with the best approximation in $L^2_{\mu}$. 
	We have the following result from \cite{Cohen2016} for a conditional weighted least-squares projection, here stated in a slightly different form.
	\begin{theorem}[\cite{Cohen2016}]\label{th:conditional-projection}
		Let $\bm x^n$ be drawn from the measure  $\bm \rho^n$ and let 
		$Q_{V_m}^{\bm x^n} u$ be the associated weighted least-squares projection of $u$.
		For any $\delta \in [0,1)$ and $\eta \in [0,1]$ such that \eqref{bound-PZ} holds, 
		\begin{equation}\label{bound-conditional-projection}
		\mathbb{E}(\Vert u - {Q}^{\bm x^n,C}_{V_m}u \Vert^2) \le (1+(1-\delta)^{-1}) \Vert u - P_{V_m}u \Vert^2 + \eta \Vert u \Vert^2 ,
		\end{equation}
		where ${Q}^{\bm x^n,C}_{V_m}u = Q_{V_m}^{\bm x^n} u$ if $Z_{\bm x^n}\le \delta$ and $0$ otherwise.
	\end{theorem}
	\begin{proof}
		We have 
		$$
		\mathbb{E}(\Vert u - {Q}^{\bm x^n,C}_{V_m}u \Vert^2)  = 
		\mathbb{E}(\Vert u -Q_{V_m}^{\bm x^n} u \Vert^2 \mathds{1}_{Z_{\bm x^n} \le \delta}) 
		+  \Vert u  \Vert^2  \mathbb{E}(\mathds{1}_{Z_{\bm x^n} > \delta}),
		$$
		with $\mathbb{E}(\mathds{1}_{Z_{\bm x^n} > \delta}) = \mathbb{P}(Z_{\bm x^n} > \delta) \le \eta$.
		Then using Lemma \ref{lem:quasi-optim-under-condition} and \eqref{mean-discrete-norm}, we have 
		\begin{equation*}
		\begin{aligned}
		\mathbb{E}(\Vert u - Q_{V_m}^{\bm x^n} u \Vert^2 \mathds{1}_{Z_{\bm x^n} \le \delta}) & \le  \mathbb{E}((\Vert u - P_{V_m} u \Vert^2+ (1-\delta)^{-1} \Vert u - P_{V_m}u \Vert_{\bm x^n}^2 ) \mathds{1}_{Z_{\bm x^n} \le \delta})
		\\
		&\le \Vert u - P_{V_m} u \Vert^2  + (1-\delta)^{-1} \mathbb{E} (\Vert u - P_{V_m}u \Vert_{\bm x^n}^2  )
		\\
		&= (1+(1-\delta)^{-1}) \Vert u - P_{V_m} u \Vert^2,
		\end{aligned}
		\end{equation*}
		which concludes the proof.
	\end{proof}	
	Also, we have the following quasi-optimality property for the weighted least-squares projection associated with the distribution  $\bm \rho^n$ conditioned to the event $\{Z_{\bm x^n} \le \delta\}$.
	\begin{theorem}\label{th:conditional-expectation}
		Let $\bm x^n$ be drawn from the measure  $\bm \rho^n$ and let 
		$Q_{V_m}^{\bm x^n} u$ be the associated weighted least-squares projection of $u$. For any $\delta \in [0,1)$ and $\eta \in [0,1)$ such that \eqref{bound-PZ} holds,  
		\begin{equation}\label{quasi-optim-conditional}
		\mathbb{E}(\Vert u - Q_{V_m}^{\bm x^n} u\Vert^2 \vert Z_{\bm x^n} \le \delta) \le ( 1 + (1-\delta)^{-1} (1-\eta)^{-1}) \Vert u - P_{V_m} u \Vert^2.
		\end{equation}
	\end{theorem}
	\begin{proof}
		From Lemma \eqref{lem:quasi-optim-under-condition}, we have that 
		$$
		\mathbb{E}(\Vert u - Q_{V_m}^{\bm x^n}u \Vert^2 \vert Z_{\bm x^n} \le \delta) \le \Vert u - P_{V_m}u \Vert^2 + 
		(1-\delta)^{-1} \mathbb{E}(\Vert u - P_{V_m} u\Vert^2_{\bm x^n} \vert Z_{\bm x^n} \le \delta),
		$$
		and 
		$$
		\mathbb{E}(\Vert u - P_{V_m}u \Vert^2_{\bm x^n} \vert Z_{\bm x^n} \le \delta) \le \mathbb{E}(\Vert u - P_{V_m} u\Vert^2_{\bm x^n}) \mathbb{P}( Z_{\bm x^n} \le \delta)^{-1},
		$$
		and we conclude by using $\mathbb{P}( Z_{\bm x^n} \le \delta) \ge 1-\eta$ and the property \eqref{mean-discrete-norm}.
	\end{proof}

	\subsection{Optimal sampling measure}
	An inequality of the form \eqref{bound-PZ}
	can be obtained by concentration inequalities. A suitable sampling distribution can then be obtained by an optimization of the obtained upper bound.  
	An optimal choice for $w$ based on matrix Chernoff inequality is derived in \cite{Cohen2016} and given by 
	\begin{equation}\label{optimal-w}
	w(x)^{-1}= \frac{1}{m}\sum_{j=1}^m\varphi_j(x)^2 = \frac{1}{m} \Vert \bm{\varphi}(x) \Vert^2_2.
	\end{equation}
	Using this distribution, we obtain the following result, for which we provide a sketch of proof following \cite{Cohen2016}. The result is here provided in a slightly more general form than in \cite{Cohen2016}.
	\begin{theorem}
		\label{th:condition-sample-size}
		Let $\eta \in [0,1)$ and $\delta \in [0,1)$. Assume $\bm{x}^n$ is drawn from the product measure $\bm \rho^n = \rho^{\otimes n}$, with $\rho$ having the density \eqref{optimal-w} with respect to $\mu$. If the sample size $n$ is such that\footnote{Note that the constant in the condition \eqref{numberofsamples} differs from the one given in the reference  \cite{Cohen2016} for $\delta =1/2$, which was incorrect.}
		\begin{equation}
		\label{numberofsamples}
		n \ge n(\delta,\eta, m) := d_{\delta}^{-1} m \log\left( {2m}{\eta^{-1}}\right),
		\end{equation}
		with $d_{\delta} := -\delta + (1+\delta)\log(1+\delta)$, then $Z_{\bm{x}^n} = \Vert \bm{G}_{\bm{x}^n}- \bm{I}  \Vert_2$ satisfies 
		\eqref{bound-PZ}.
	\end{theorem}
	\begin{proof}
		We have $\bm{G}_{\bm{x}^n} = \frac{1}{n} \sum_{i=1}^n \bm{A} _i$ where the $\bm{A}_i = w(x^i)\bm{\varphi}(x^i)\otimes \bm{\varphi}(x^i)$ are random matrices such that $\mathbb{E}(\bm{A}_i) = \bm I$ and  $\Vert \bm A_i \Vert_2 = w(x^i) \Vert \bm{\varphi}(x^i)\Vert_2^2 = m$. The matrix Chernoff inequality from \cite[Theorem 5.1]{tropp2012user} gives that the minimal and maximal eigenvalues of $\bm{G}_{\bm{x}^n} - \bm I$ satisfy 
		$$
		\mathbb{P}(\lambda_{min}(\bm{G}_{\bm{x}^n} - \bm I) < - \delta)  \vee \mathbb{P}(\lambda_{max}(\bm{G}_{\bm{x}^n} - \bm I) > \delta) \le m \exp(-n d_\delta /m).
		$$  
		Under the condition \eqref{numberofsamples}, we have that $m \exp(-n d_\delta /m) \le {\eta}/{2}$ and 
		using a union bound, we deduce \eqref{bound-PZ}. 
	\end{proof}
	\begin{remark}
		Note that $d_\delta \le \delta^2 $. Then  a sufficient condition for satisfying the condition \eqref{numberofsamples} is 
		$n\ge \delta^{-2} m \log\left( {2m}{\eta^{-1}}\right).$
	\end{remark}
	\begin{remark}
		The quantile function of $Z_{\bm x^n}$ is defined for 
		$t \in [0,1]$ by $F_{Z_{\bm x^n}}^-(t) = \inf \{\delta : F_{Z_{\bm x^n}}(\delta) \ge t\}$, where $F_{Z_{\bm x^n}}$ is the cumulative density function of the random variable $Z_{\bm x^n}$.\\
		For given $n$ and $\eta$, $F_{Z_{\bm x^n}}^-(1-\eta) $ is the minimal $\delta$ such that $\eqref{bound-PZ} $ is satisfied. Denoting by $\delta_c(\eta,n) = \min \{\delta : n \ge n(\delta,\eta, m) \} $, we clearly have $F_{Z_{\bm x^n}}^-(1-\eta) \le \delta_c(\eta,n) $.
		The closer $\delta_c$ is from $F_{Z_{\bm x^n}}^-(1-\eta)$, the sharper the condition on the sample size $n$ is for satisfying \eqref{bound-PZ}.
	\end{remark}
	Theorem \ref{th:condition-sample-size} states that using the optimal sampling density \eqref{optimal-w}, a stable projection of $u$ is obtained with a sample size in $\mathcal{O}(m\log(m \eta^{-1}))$ with high probability. Note that a small probability $\eta$, and therefore a large sample size $n$, may be required for controlling the term $\eta \Vert u\Vert^2$ in the error bound \eqref{bound-conditional-projection} for the conditional projection, or for obtaining a quasi-optimality property \eqref{quasi-optim-conditional} in conditional expectation  with a quasi-optimality constant close to $1+(1-\delta)^{-1}.$ This will be improved in the next section by proposing a new distribution (obtained by resampling, conditioning and subsampling) allowing to obtain stability of the empirical Gram matrix with very high probability and a moderate sample size.
	
	\section{Boosted optimal weighted least-squares method}\label{sec:boosted}
	We here propose an improved  weighted least-squares method by proposing distributions over $\mathcal{X}^n$ having better properties than $\bm \rho^n = \rho^{\otimes n}.$ The function  $w$ defining the weighted least-squares projections will always be taken such that $w^{-1}$ is the density of the optimal sampling measure $\rho$ with respect to the reference measure $\mu.$
	
	\subsection{Resampling and conditioning}
	The first improvement consists in drawing $M$ independent samples 
	$\{\bm{x}^{n,i}\}_{i=1}^M$, with $\bm{x}^{n,i} = (x^{1,i}, \hdots, x^{n,i})$, from the distribution $\bm{\rho}^n$, and then in selecting a sample $\bm{x}^{n,\star}$ which satisfies
	\begin{equation}
	\label{minimumsample}
	\Vert \bm{G}_{\bm{x}^{n,\star}}- \bm{I} \Vert_2  = \min_{1 \le i \le M} \Vert\bm{G}_{\bm{x}^{n,i}}- \bm{I} \Vert_2,
	\end{equation}
	where $\bm{G}_{\bm{x}}$ denotes the empirical Gram matrix associated with a sample $\bm x $ in $\mathcal{X}^n$. If several  
	samples  $\bm{x}^{n,i}$ are solutions of the minimization problem, $\bm{x}^{n,\star}$ is selected at random among the minimizers.  
	We denote by $\bm{\rho}^{n,\star}$ the probability measure of $\bm{x}^{n,\star}$.
	The probability that the stability condition $Z_{\bm{x}^{n,\star}} = \Vert \bm{G}_{\bm{x}^{n,\star}}- \bm{I} \Vert_2 \le \delta$ is verified can be made arbitrarily high, playing on $M$, as it is shown in the following lemma (whose proof is trivial).
	\begin{lemma}
		\label{lem:resampling-proba}
		For any $\delta \in [0,1)$ and $\eta \in (0,1)$, if  $n$ satisfies \eqref{numberofsamples}, then
		\begin{equation}
		\mathbb{P}(Z_{\bm{x}^{n,{\star}}}   \le \delta) \ge 1 - \eta^M.
		\end{equation}
	\end{lemma}
	Therefore, we can choose a probability $\eta$ arbitrary close to $1$, so that 
	the condition \eqref{numberofsamples} does not require a large sample size $n$, and still obtain the stability condition with a probability at least $1-\eta^M$ which can be made arbitrarily close to $1$ by choosing a sufficiently large $M$.
	Even if $\bm \rho^n$ has a product structure, for $M>1,$ the distribution $\bm{\rho}^{n,\star} $ does not have a product structure, i.e.  the components of $ \bm{x}^{n,\star} = (x^{1,\star},\hdots,x^{n,\star})$ are not independent, and does not satisfy the assumptions of Theorems \ref{th:conditional-projection} and \ref{th:conditional-expectation}. In particular $\mathbb{E}( \bm{G}_{\bm{x}^{n,\star}})$ may not be equal to $\bm{I}$ and in general, $\mathbb{E}(\Vert v \Vert^2_{\bm x^{n,\star}}) \neq \Vert v \Vert^2$ for an arbitrary function $v$ when $M>1$. Therefore, a new analysis of the resulting weighted least-squares projection is required.
	\begin{remark}
		Note that 
		since the function $\bm x \mapsto \Vert\bm{G}_{\bm{x}}- \bm{I} \Vert_2$ defined on $\mathcal{X}^{n}$ is invariant through permutations  of the components of $\bm x$, we have that the components of  $\bm{x}^{n,\star}$ have the same marginal distribution.
	\end{remark}
	In order to ensure that the stability property  is verified almost surely
	we consider a sample $\widetilde{\bm x}^n$ from the distribution  $\widetilde{\bm \rho}^n$ of $\bm x^{n,\star}$ conditioned on the event
	\begin{equation}
	A_{\delta} = \{ \Vert \bm{G}_{\bm{x}^{n,{\star}}} - \bm{I} \Vert_2 \le \delta\},\label{eventAdelta}
	\end{equation}
	which is such that for  any function $f$, $\mathbb{E}(f(\widetilde{\bm x}^n)) = \mathbb{E}(f(\bm x^{n,\star}) \vert A_\delta).$ 
	A sample $\widetilde {\bm x}^n$ from the distribution $\widetilde{\bm \rho}^n$ is obtained by a simple rejection method, which consists in drawing samples $\bm x^{n,\star}$ from the distribution $\bm\rho^{n,\star}$ until $A_\delta$ is satisfied. It follows that $\mathbb{P}(Z_{\widetilde{\bm{x}}^{n}}   \le \delta) = 1 $.
	\begin{remark}
		Let $J$ be the number of trials necessary to get a sample $\bm x^{n,\star}$ verifying the stability condition $A_{\delta}$. This random variable $J$ follows a geometric distribution with a probability of success $\mathbb{P}(A_{\delta})$. Therefore $J$ is almost surely finite and 
		\begin{equation}
		\mathbb{P}(J \ge k) = (1-\mathbb{P}(A_{\delta}))^k,
		\end{equation}
		i.e. the probability to have $J$ greater than $k$ decreases exponentially with $k$. An other property of the geometric distribution is that $\mathbb{E}(J) = \frac{1}{1-\eta}$, such that the average number of trials increases when $\eta$ tends to $1$, in particular we have $\mathbb{E}(J) = 2$ for $\eta = 0.5$ and $\mathbb{E}(J) = 100$ for $\eta = 0.99$.
	\end{remark}
	Now we provide a result on the distribution of $\widetilde {\bm x}^n$ which will be later used for the analysis of the corresponding least-squares projection. 
	
	\begin{lemma}\label{lem:minimumexpectation}
		Let $\widetilde{\bm{x}}^n$ be a sample following the distribution  $\widetilde{\bm{\rho}}^n$, which is the distribution $\bm \rho^{n,\star} $ conditioned on the event $A_\delta $ defined by \eqref{eventAdelta}.  Assume that $n \ge n(\delta,\eta, m)$ for some $\eta \in (0,1)$ and $\delta \in (0,1).$
		Then for any function $v$ in $L^2_{\mu}$ and any $0<\varepsilon\le 1$, 
		\begin{equation}\label{norm-tildexn-bound-epsilon}
		\mathbb{E}(\Vert v \Vert_{\widetilde{\bm{x}}^n}^2) \le C(\varepsilon,M)(1 - \eta^M)^{-1} \mathbb{E}(\Vert v \Vert^{2/\varepsilon}_{\bm{x}^{n}})^{\varepsilon},
		\end{equation}
		with $$C(\varepsilon,M) = M  \frac{(1- \varepsilon)^{1 - \varepsilon}}{(M -\varepsilon)^{1 - \varepsilon}} \le M. $$
		In particular, for $\varepsilon=1$, 
		\begin{equation}\label{norm-tildexn-bound}
		\mathbb{E}(\Vert v \Vert_{\widetilde{\bm{x}}^n}^2) \le M(1 - \eta^M)^{-1} \Vert v \Vert^2.
		\end{equation}
		Also, if $\Vert v \Vert_{\infty,w} = \sup_{x\in \mathcal{X}} w(x)^{1/2} \vert v(x) \vert<\infty$, 
		\begin{equation}\label{norm-tildexn-bound-Linf}
		\mathbb{E}(\Vert v \Vert_{\widetilde{\bm{x}}^n}^2) \le C(\varepsilon,M)(1 - \eta^M)^{-1} \Vert v \Vert_{\infty,w}^{2-2\varepsilon}  \Vert v \Vert^{2\varepsilon}.
		\end{equation}
	\end{lemma}
	\begin{proof}
		See appendix. 
	\end{proof}
	\begin{corollary}
		Let $\widetilde{\bm{x}}^n$ be a sample following the distribution $\widetilde{\bm{\rho}}^n$ and assume that $n \ge n(\delta,\eta, m)$ for some $\eta \in (0,1)$ and $\delta \in (0,1)$. For any $v \in L_{\mu}^2$, the weighted least-squares projection $Q_{V_m}^{\widetilde{\bm{x}}^n}v$ associated with the sample $\widetilde{\bm{x}}^n$ satisfies
		\begin{equation}
		\mathbb{E}(\Vert Q_{V_m}^{\widetilde{\bm{x}}^n}v \Vert^2) \le (1-\delta)^{-1}M(1 - \eta^M)^{-1} \Vert v \Vert^2.
		\end{equation}
	\end{corollary}
	\begin{proof}
		Since $Q_{V_m}^{\widetilde{\bm{x}}^n}v \in V_m$, we have that
		\begin{equation}
		\Vert Q_{V_m}^{\widetilde{\bm{x}}^n}v\Vert^2 \le (1-\delta)^{-1} \Vert Q_{V_m}^{\widetilde{\bm{x}}^n} v \Vert^2_{\widetilde{\bm{x}}^n} \le (1-\delta)^{-1} \Vert v \Vert^2_{\widetilde{\bm{x}}^n},
		\end{equation}	
		where we have used the fact that $Q_{V_m}^{\widetilde{\bm{x}}^n}$ is an orthogonal projection with respect to the semi-norm $\Vert \cdot \Vert_{\widetilde{\bm{x}}^n}$. Taking the expectation and using \ref{norm-tildexn-bound}, we obtain
		\begin{equation}
		\mathbb{E}(\Vert Q_{V_m}^{\widetilde{\bm{x}}^n} v\Vert^2) \le (1-\delta)^{-1}M(1 - \eta^M)^{-1} \Vert v \Vert^2.
		\end{equation}
	\end{proof}
	\begin{theorem}
		\label{th:bls_accuracy}
		Let $\widetilde{\bm{x}}^n$ be a sample following the distribution  $\widetilde{\bm{\rho}}^n$ and assume that $n \ge n(\delta,\eta, m)$
		for some $\eta \in (0,1)$ and $\delta \in (0,1).$ The  
		weighted least-squares projection $ Q_{V_m}^{\widetilde {\bm x}^n}u$  associated with the sample $\widetilde{\bm{x}}^n$ satisfies the quasi-optimality property
		\begin{equation}\label{bls_accuracy-quasi-optim}
		\mathbb{E}(\Vert u - Q_{V_m}^{\widetilde {\bm x}^n}u \Vert^2) \le (1+ (1-\delta)^{-1}(1 - \eta^M)^{-1} M)\Vert u - P_{V_m}u \Vert^2.
		\end{equation}
		Also, assuming $\Vert u \Vert_{\infty,w} \le L$, we have 
		\begin{equation}\label{quasi-optim-epsilon-inf}
		\mathbb{E}(\Vert u - Q_{V_m}^{\widetilde {\bm x}^n}u \Vert^2) \le \left(\Vert u - P_{V_m}u \Vert^2 + (1-\delta)^{-1}(1 - \eta^M)^{-1}D(M,L,m,\Vert u - P_{V_m}u \Vert^2)\right) 
		\end{equation}
		where for all $\alpha \ge 0$, $D(M,L,m, \alpha) = \inf_{0 < \varepsilon \le 1}\widetilde{D}(M,L,m, \alpha, \varepsilon)$, with 
		$$\widetilde{D}(M,L,m, \alpha, \varepsilon) := C(\varepsilon,M) (L(1+c_m))^{2-2\varepsilon}\alpha^{\varepsilon}.$$
		Here, $C(\varepsilon,M)$ is the constant defined in Lemma \ref{lem:minimumexpectation} and $c_m$ the supremum of $\Vert P_{V_m} v \Vert_{\infty,w}$ over functions $v$ such that $\Vert v \Vert_{\infty,w} \le 1$.
	\end{theorem}
	\begin{proof}
		From Lemma \ref{lem:quasi-optim-under-condition}, we have that   
		\begin{equation*}
		\Vert u - Q_{V_m}^{\widetilde {\bm x}^n}u \Vert^2  \le  \Vert u - P_{V_m}u \Vert^2 + (1-\delta)^{-1}\Vert u -P_{V_m}u \Vert_{\widetilde{\bm{x}}^n}^2
		\end{equation*}
		holds almost surely, 
		and from Lemma \ref{lem:minimumexpectation}, we have that
		\begin{equation*}
		\mathbb{E}(\Vert u - P_{V_m}u \Vert_{\widetilde{\bm{x}}^n}^2) \le C(\varepsilon,M)(1 - \eta^M)^{-1}\mathbb{E}(\Vert u - P_{V_m} u \Vert^{2/\varepsilon}_{\bm{x}^{n}})^{\varepsilon}
		\end{equation*}
		for all $\varepsilon \in (0,1].$ 
		Combining the above inequalities and then taking the infimum over $\varepsilon,$ we obtain  
		\begin{equation}\label{quasi-optim-epsilon}
		\mathbb{E}(\Vert u - Q_{V_m}^{\widetilde {\bm x}^n}u \Vert^2) \le \Vert u - P_{V_m}u \Vert^2 + (1-\delta)^{-1}(1 - \eta^M)^{-1}\inf_{0 < \varepsilon \le 1} C(\varepsilon,M)  \mathbb{E} \left(\Vert u - P_{V_m} u \Vert^{2/\varepsilon}_{\bm{x}^n} \right)^{\varepsilon}.
		\end{equation}
		The particular case $\varepsilon=1$ yields \eqref{bls_accuracy-quasi-optim}. 
		The second property  \eqref{quasi-optim-epsilon-inf} is simply deduced from \eqref{quasi-optim-epsilon} by using the property \eqref{norm-tildexn-bound-Linf} of Lemma \ref{lem:minimumexpectation} and by noting that $\Vert u - P_{V_m}u \Vert_{\infty,w} \le (1+c_m) \Vert u
		\Vert_{\infty,w}.$
	\end{proof}
	\begin{remark}
		\label{rem:bound_cm}
		The constant $c_m$ in Theorem \ref{th:bls_accuracy} is such that $c_m \le m$. Indeed, $P_{V_m} v(x) = \sum_{i=1}^m a_i \varphi_i(x)$ with $$a_i = (v,\varphi_i) = \int v(x) \varphi_i(x) d\mu(x) = \int v(x) \varphi_i(x) w(x)d\rho(x),$$ so that 
		\begin{equation*}
		\begin{aligned}
		\vert a_i \vert & \le \Vert v \Vert_{\infty,w} \int \vert \varphi_i(x) \vert w(x)^{1/2} d\rho(x) \le \Vert v \Vert_{\infty,w} (\int \varphi_i(x)^2w(x)d\rho(x))^{1/2} = \Vert v \Vert_{\infty,w},
		\end{aligned}
		\end{equation*}
		where we have used Cauchy-Schwarz inequality. Therefore, 
		\begin{equation*}
		\begin{aligned}
		\Vert P_{V_m} v \Vert_{\infty,w} & \le \Vert v \Vert_{\infty,w} \sup_{x\in \mathcal{X}} w(x)^{1/2} \sum_{i=1}^m \vert \varphi_i(x) \vert \\
		& \le  \Vert v \Vert_{\infty,w} \sup_{x\in \mathcal{X}} w(x)^{1/2} m^{1/2} (\sum_{i=1}^m  \varphi_i(x)^2)^{1/2} = m \Vert v \Vert_{\infty,w}.
		\end{aligned}
		\end{equation*}
	\end{remark}
	\begin{remark}
		\textbf{About the constant $D(L,M,m, \alpha)$.}\\
		The value of $\varepsilon$ that minimizes $\widetilde{D}(M,L,m,\alpha, \varepsilon)$ can be shown to be
		$$\varepsilon^{\star}  =\frac{1 - M W(\exp(-1 -2\log(L(1+c_m)) + 2\log(\alpha))}{1 -  W(\exp(-1 -2\log(L(1+c_m)) + 2\log(\alpha)) }$$
		where $x \mapsto W(x)$ is the Lambert function which represents the solution $y$ of the equation $ y\exp(y) = x$.\\
		In Figures \ref{fig:borne_diff_m} and \ref{fig:borne_diff_M}, we illustrate the fact that the property \eqref{quasi-optim-epsilon-inf} may improve \eqref{bls_accuracy-quasi-optim} if $D(M,L,m,\Vert u - P_{V_m}u \Vert^2) \le M\Vert u - P_{V_m}u \Vert^2$, for some conditions on $M,L,m$. The legend "Initial bound" refers to the bound presented in \ref{bls_accuracy-quasi-optim}, and the legend "Improved bound" refers to the bound presented in \ref{quasi-optim-epsilon-inf}.
		\begin{center}
			\begin{figure}
				\begin{subfigure}[b]{0.49\textwidth}
					\centering
					\includegraphics[scale =0.8]{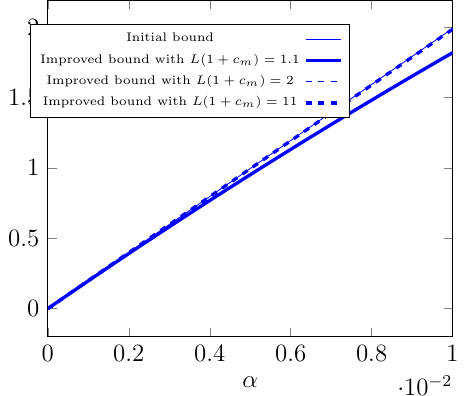}
					\caption{$M=100$, $\delta = 0.5$, $\eta = 0.01^{1/M}$ }
					\label{fig:borne_diff_m}	
				\end{subfigure}
				\begin{subfigure}[b]{0.49\textwidth}
					\centering
					\includegraphics[scale =0.8]{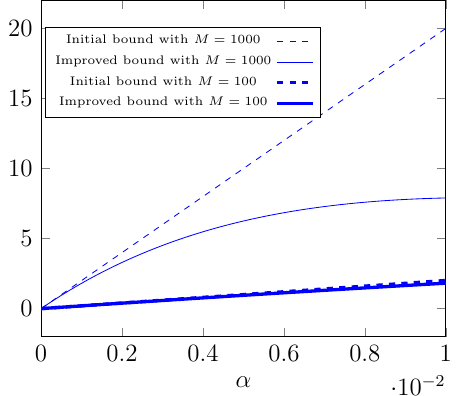}
					\caption{$L(1+c_m)=2$, $\delta = 0.5$, $\eta = 0.01^{1/M}$ }
					\label{fig:borne_diff_M}
				\end{subfigure}
				\caption{Improvement of the bound for different values of $m$ and $M$.}
			\end{figure}
		\end{center}
		The $x$-axis represents the best approximation error, so that for a given $L$ and a given $c_m$, the left part of the curve corresponds to functions $u$ which can be well approximated in $V_m$ whereas on the contrary, the right part of the curve corresponds to functions which are not well approximated in $V_m$. We observe that the bound from \ref{quasi-optim-epsilon-inf} improves the bound obtained with \ref{bls_accuracy-quasi-optim} when $M$ is high $(M \ge 1000)$, and when $L(1+c_m)$ is small $(L(1+c_m) \le 1.1)$.
	\end{remark}
	\subsection{Subsampling}
	\label{sec:subsampling}
	Although the resampling enables us to choose $\delta$ and $\eta$ such that $n$ is smaller than with the initial strategy from \cite{Cohen2016}, the value of $n$ may still be high compared to an interpolation method. Therefore, to further decrease the sample size, for each generated sample $\widetilde{\bm{x}}^n$, we propose to select a subsample which still verifies the stability condition.\\
	We start with a sample $\widetilde{\bm{x}}^n = (\widetilde{x}^1,\hdots,\widetilde{x}^n)$ satisfying $  \Vert \bm{G}_{\widetilde{\bm{x}}^n} - \bm{I} \Vert_2 \le \delta$ and then select a subsample $\widetilde{\bm{x}}^n_K = (\widetilde{x}^k)_{k \in K}$ with $K \subset \{1, \hdots, n\}$ such that the empirical Gram matrix $\bm{G}_{\widetilde{\bm{x}}^n_K} = \frac{1}{\# K}\sum_{k \in K} w(\widetilde{x}^k)\bm{\varphi}(\widetilde{x}^k)\otimes \bm{\varphi}(\widetilde{x}^k)$ still satisfies $$\Vert \bm{G}_{{\widetilde{\bm{x}}^n}_K} - \bm{I}\Vert_2 \le \delta.$$ In practice, the set $K$ is constructed by a greedy procedure. We start with $K = \{1, \hdots, n\}$. Then at each step of the greedy procedure, we select $k^{\star}$ in $K$ such that
	\begin{equation}
	\Vert \bm{G}_{{\widetilde{\bm{x}}^n}_{K \setminus \{k^{\star}\}}} - \bm{I} \Vert_2 = \min_{k \in K} \Vert \bm{G}_{{\widetilde{\bm{x}}^n}_{K \setminus \{k\}}} - \bm{I} \Vert_2.
	\end{equation}
	If $\Vert \bm{G}_{{\widetilde{\bm{x}}^n_{K \setminus \{k^{\star}\}}}}  - \bm{I} \Vert_2 \le \delta$ and $\# K > n_{min}$ then $k^{\star}$ is removed from $K$. Otherwise, the algorithm is stopped. We denote by $\widetilde{\bm{\rho}}^n_K$ the distribution of the sample  $\widetilde{\bm{x}}^n_K$ produced by this greedy algorithm.
	
	\begin{theorem}
		\label{th:s-BLS-accuracy}
		Assume $n \ge n(\delta,\eta, m)$ for some $\eta \in (0,1)$ and $\delta \in (0,1),$ and let 
		$\widetilde{\bm{x}}^n_K$ be a sample produced by the greedy algorithm with $\# K \ge n_{min}$.
		The weighted least-squares projection  
		$ Q_{V_m}^{\widetilde{\bm{x}}^n_K}u$  associated with the sample $\widetilde{\bm{x}}^n_K$ satisfies the quasi-optimality property
		\begin{equation}\label{quasi-optim-greedy}
		\mathbb{E}(\Vert u - Q_{V_m}^{\widetilde{\bm{x}}^n_K}u \Vert^2) \le (1+  \frac{n}{n_{min}}(1-\delta)^{-1}(1 - \eta^M)^{-1} M)\Vert u - P_{V_m}u \Vert^2.
		\end{equation}
		Also, assuming $\Vert u \Vert_{\infty,w} \le L$ , we have 
		\begin{equation}\label{quasi-optim-epsilon-inf-greedy}
		\mathbb{E}(\Vert u - Q_{V_m}^{\widetilde {\bm x}^n_K}u \Vert^2) \le \Vert u - P_{V_m}u \Vert^2 + \frac{n}{n_{min}}(1-\delta)^{-1}(1 - \eta^M)^{-1}D(M,L,m,\Vert u - P_{V_m}u \Vert^2)
		\end{equation}
		where $D(M,L,m,\Vert u - P_{V_m}u \Vert^2)$ is defined in Theorem \ref{th:s-BLS-accuracy}.
	\end{theorem}
	\begin{proof} 
		Since $Z_{\widetilde{\bm{x}}^n_K}\le \delta$, from Lemma \ref{lem:quasi-optim-under-condition}, we have that for any $v \in V_m$, the least-squares projection associated with ${\widetilde{\bm{x}}^n_K}$ satisfies 
		\begin{equation}
		\begin{aligned}
		\Vert u - Q_{V_m}^{\widetilde{\bm{x}}^n_K}u \Vert^2 & \le \Vert u - P_{V_m} u \Vert^2 + (1-\delta)^{-1} \Vert  u - P_{V_m} u  \Vert_{\widetilde{\bm{x}}^n_K}^2\\
		& \le \Vert u -  P_{V_m} u \Vert^2 + (1-\delta)^{-1}\frac{n}{\#K} \Vert u -  P_{V_m} u  \Vert_{\widetilde{\bm{x}}^n}^2,\\
		\end{aligned}
		\end{equation}
		where the second inequality simply results from 
		$$ \Vert v \Vert_{\widetilde{\bm{x}}^n_K}^2 = \frac{1}{\# K}\sum_{k\in K}w(\widetilde x^k) v(\widetilde x^k)^2 \le \frac{1}{\# K} \sum_{k =1}^n w(\widetilde x^k) v(\widetilde x^k)^2 =  \frac{n}{\# K} \Vert v \Vert_{\widetilde {\bm x}^n}^2.$$ 
		Therefore, since $\#K \ge n_{min}$, we obtain from Lemma \ref{lem:minimumexpectation} that
		\begin{equation*}\label{quasi-optim-epsilon-greedy}
		\mathbb{E}(\Vert u - Q_{V_m}^{\widetilde{\bm{x}}^n_K}u \Vert^2) \le \Vert u - P_{V_m}u \Vert^2 + \frac{n}{n_{min}} (1-\delta)^{-1}(1 - \eta^M)^{-1}\inf_{0 < \varepsilon \le 1} C(\varepsilon,M) \mathbb{E} \left(\Vert u - P_{V_m} u \Vert^{\frac{2}{\varepsilon}}_{\bm{x}^n} \right)^{\varepsilon}.
		\end{equation*}
		The particular case $\varepsilon=1$ yields the first property. For the second property, the proof follows the one of 
		the property \eqref{quasi-optim-epsilon-inf} in Theorem \ref{th:bls_accuracy}.
	\end{proof}
	\begin{corollary}
		\label{cor:bound_optimal_proj_with_subs}
		Assume $n \ge n(\delta,\eta, m)$ for some $\eta \in (0,1)$ and $\delta \in (0,1),$ and let 
		$\widetilde{\bm{x}}^n_K$ be a sample produced by the greedy algorithm with $\# K \ge n_{min}$.
		The weighted least-squares projection  
		$ Q_{V_m}^{\widetilde{\bm{x}}^n_K}u$  associated with the sample $\widetilde{\bm{x}}^n_K$ satisfies
		\begin{equation}
		\mathbb{E}(\Vert Q_{V_m}^{\widetilde{\bm{x}}^n_K}v \Vert^2) \le (1-\delta)^{-1}M(1 - \eta^M)^{-1}\frac{n}{n_{min}} \Vert v \Vert^2.
		\end{equation}
	\end{corollary}
	\begin{proof}
		Since $Q_{V_m}^{\widetilde{\bm{x}}^n_K}v \in V_m$, we have that
		\begin{equation}
		\Vert Q_{V_m}^{\widetilde{\bm{x}}^n_K}v\Vert^2 \le (1-\delta)^{-1} \Vert Q_{V_m}^{\widetilde{\bm{x}}^n_K} v \Vert^2_{\widetilde{\bm{x}}^n_K} \le (1-\delta)^{-1} \Vert  v \Vert^2_{\widetilde{\bm{x}}^n_K} \le (1-\delta)^{-1}\frac{n}{\#K} \Vert v \Vert^2_{\widetilde{\bm{x}}^n},
		\end{equation}	
		where we have used the fact that $Q_{V_m}^{\widetilde{\bm{x}}^n_K}$ is an orthogonal projection with respect to the semi-norm $\Vert \cdot \Vert_{\widetilde{\bm{x}}^n_K}$ and the fact that $ \Vert v \Vert_{\widetilde{\bm{x}}^n_K}^2 \le \frac{n}{\# K} \Vert v \Vert_{\widetilde {\bm x}^n}$. Taking the expectation, using \ref{norm-tildexn-bound} and assuming that $\# K \le n_{min}$, we obtain
		\begin{equation}
		\mathbb{E}(\Vert Q_{V_m}^{\widetilde{\bm{x}}^n_K} v\Vert^2) \le (1-\delta)^{-1}M(1 - \eta^M)^{-1}\frac{n}{n_{min}} \Vert v \Vert^2.
		\end{equation}
	\end{proof}
	If we set $n_{min} =m$, it may happen that the algorithm runs until $\#K =m$, the interpolation regime. Choosing $n \ge n(\delta,\eta,m)$ then yields a quasi-optimality constant depending on $\log(m)$. It has to be compared with the optimal behaviour of the Lebesgue constant for polynomial interpolation in one dimension.  
	If we choose  $n_{min} = n / \beta$ for some fixed $\beta \ge 1$ independent of $m$, then we have $\frac{n}{n_{min}} \le  \beta$ and a quasi-optimality constant independent of $m$ in \eqref{quasi-optim-epsilon-greedy}, but the algorithm may stop before reaching the interpolation regime $(n=m)$.
	\begin{remark}
		Concerning, the greedy subsampling, a direct approach to remove a point is to calculate the norm of $\Vert \bm{G}_{{\widetilde{\bm{x}}^n}_{K \setminus \{k\}}} - \bm{I} \Vert_2$ for each $k \in K$. However, it involves to calculate this norm for $\#K$ points and this each time a point is removed. In the Algorithm \ref{alg:greedy_strategy} we present a method which enables us to choose $k^{\star}$ by performing simple matrix multiplications. Indeed, knowing the eigenvalues of a symmetric matrix, there exists bounds on the eigenvalues of a rank-one update of this matrix (see \cite{Bunch1978} and \cite{Golub1973} for more details, as well as the more recent results from \cite{Benasseni2011} that we use in practice).
	\end{remark}
	\section{The noisy case}
	\label{sec:noisy_case}
	We here consider the case where the observations are polluted with a noise, which is modeled by a random variable $e$. More precisely the observed data take the form
	$$y^i = u(\widetilde{x}^i) + e^i, $$
	where $\{e^i \}_{i \in K}$ are i.i.d realizations of the random variable $e$ and $\{\widetilde{x}^i\}_{i \in K} = \widetilde{\bm{x}}^n_K$ are the points built with the boosted least-squares method.  
	We assume the noise is independent from $\widetilde{\bm{x}}^n$ and centered $\mathbb{E}(e) = 0$ and with bounded variance $\sigma^2 =  \mathbb{E}(\vert e \vert^2 ) < \infty$. More general cases could be considered as in \cite{Cohen2013} or \cite{MiglioNobileTempone2015}.\\
	The weighted discrete least-squares projection of $u$ over $V_m$ is defined by
	\begin{equation}
	u^{\widetilde{\bm{x}}^n_K} := \arg\min_{v \in {V_m}} \frac{1}{\#K} \sum_{i \in K} w(\widetilde{x}^i)(y^i- v(\widetilde{x}^i)) ^2.
	\end{equation}
	\begin{theorem}
		Assume $n \ge n(\delta,\eta, m)$ for some $\eta \in (0,1)$ and $\delta \in (0,1),$ and let 
		$\widetilde{\bm{x}}^n_K$ be a sample produced by the greedy algorithm with $\# K \ge n_{min}$.
		The weighted least-squares projection  
		$u^{\widetilde{\bm{x}}^n_K}$  associated with the sample $\widetilde{\bm{x}}^n_K$ and the data affected by the noise $e$, satisfies 
		\begin{equation}\label{quasi-optim-greedy-noisy}\small
		\mathbb{E}(\Vert u - u^{\widetilde{\bm{x}}^n_K} \Vert^2) \le (1+  \frac{2n}{n_{min}}(1-\delta)^{-1}(1 - \eta^M)^{-1} M)\Vert u - P_{V_m}u \Vert^2 + \frac{2\sigma^2mn}{n_{min}^2}(1-\delta)^{-1}(1-\eta^M)^{-1}M .
		\end{equation}
	\end{theorem}
	\begin{proof}
		Thanks the Pythagorean equality it holds,
		\begin{equation}
		\begin{aligned}
		\Vert u - u^{\widetilde{\bm{x}}^n_K} \Vert^2 & = \Vert u - P_{V_m}u \Vert^2 + \Vert P_{V_m}u - u^{\widetilde{\bm{x}}^n_K} \Vert^2\\
		& = \Vert u - P_{V_m}u \Vert^2 + \Vert Q_{V_m}^{\widetilde{\bm{x}}^n_K}(P_{V_m}u - u)+ Q_{V_m}^{\widetilde{\bm{x}}^n_K}u  - u^{\widetilde{\bm{x}}^n_K} \Vert^2\\
		\end{aligned}
		\end{equation}
		where $Q_{V_m}^{\widetilde{\bm x}^n_K}u$  is the boosted least-squares projection of the noiseless evaluations of $u$ over $V_m$.
		Then using the triangular inequality, 
		\begin{equation}
		\begin{aligned}
		\Vert u - u^{\widetilde{\bm{x}}^n_K} \Vert^2	& \le \Vert u - P_{V_m}u \Vert^2 + 2\Vert Q_{V_m}^{\widetilde{\bm{x}}^n_K}(P_{V_m}u - u)\Vert^2 + 2\Vert Q_{V_m}^{\widetilde{\bm{x}}^n_K}u  - u^{\widetilde{\bm{x}}^n_K} \Vert^2.
		\end{aligned}
		\end{equation}
		Taking the expectation and using Corollary \ref{cor:bound_optimal_proj_with_subs}, it comes
		\small		\begin{equation}
		\begin{aligned}
		\mathbb{E}(\Vert u - u^{\widetilde{\bm{x}}^n_K} \Vert^2) & \le (1 + 2\frac{n}{n_{min}}(1-\delta)^{-1}M(1-\eta^M)^{-1}) \Vert u - P_{V_m}u \Vert^2 + 2\mathbb{E}(\Vert Q_{V_m}^{\widetilde{\bm{x}}^n_K}u  - u^{\widetilde{\bm{x}}^n_K} \Vert^2).
		\end{aligned}
		\end{equation}
		Then, we note that
		\begin{equation}
		\begin{aligned}
		\Vert Q_{V_m}^{\widetilde{\bm{x}}^n_K}u - u^{\widetilde{\bm{x}}^n_K} \Vert^2 & \le \sum_{k=1}^m \vert b_k \vert^2\\
		\end{aligned}
		\end{equation}
		where $\bm{b} = (b_k)_{k=1}^m$ is solution to 
		$$\bm{G}_{\widetilde{\bm{x}}^n_K}\bm{b} = \bm{\beta}, \ \bm{\beta} := \left( \frac{1}{\#K} \sum_{i \in K} e^i w(\widetilde{x}^i) \varphi_k(\widetilde{x}^i)\right)_{1\le k \le m}. $$
		Since $\Vert \bm{G}^{-1}_{\widetilde{\bm{x}}^n_K}\Vert^2_2 \le (1-\delta)^{-1} $ it holds 
		$$\sum_{k=1}^m \vert b_k \vert^2 \le (1-\delta)^{-1} \sum_{k=1}^m \vert \beta_k \vert^2$$
		and
		\begin{equation}
		\begin{aligned}
		\sum_{k=1}^m \vert \beta_k \vert^2 = \sum_{k=1}^m \frac{1}{(\#K)^2} \sum_{i \in K}\sum_{j \in K} e^i w(\widetilde{x}^i) \varphi_k(\widetilde{x}^i) e^j w(\widetilde{x}^j) \varphi_k(\widetilde{x}^j). 
		\end{aligned}
		\end{equation}
		As $K$ is a random variable, 
		\begin{equation}
		\begin{aligned}
		\mathbb{E}\left( \frac{1}{\#K^2} \sum_{i\in K}\sum_{j\in K} e^i w(\widetilde{x}^i) \varphi_k(\widetilde{x}^i) e^j w(\widetilde{x}^j) \varphi_k(\widetilde{x}^j)\right)	= \\  \mathbb{E}( \mathbb{E}( \frac{1}{\#K^2} \sum_{i\in K}\sum_{j\in K} e^i w(\widetilde{x}^i) \varphi_k(\widetilde{x}^i) e^j w(\widetilde{x}^j) \varphi_k(\widetilde{x}^j) \vert K )) =\\
		\mathbb{E}( \frac{1}{\#K^2} \sum_{i\in K}\sum_{j\in K} \mathbb{E}( e^i w(\widetilde{x}^i) \varphi_k(\widetilde{x}^i) e^j w(\widetilde{x}^j) \varphi_k(\widetilde{x}^j) \vert K )).\\
		\end{aligned}
		\end{equation}
		By construction $K$ is independent from the noise, 
		$$ \mathbb{E}( e^i w(\widetilde{x}^i) \varphi_k(\widetilde{x}^i) e^j w(\widetilde{x}^j) \varphi_k(\widetilde{x}^j) \vert K ) =  \mathbb{E}( e^ie^j) \mathbb{E}( w(\widetilde{x}^i) \varphi_k(\widetilde{x}^i)w(\widetilde{x}^j) \varphi_k(\widetilde{x}^j) \vert K ).$$
		Therefore, for $i \neq j$ $$ \mathbb{E}( e^i w(\widetilde{x}^i) \varphi_k(\widetilde{x}^i) e^j w(\widetilde{x}^j) \varphi_k(\widetilde{x}^j) \vert K ) = 0$$
		and for $i=j$ $$ \mathbb{E}( e^i w(\widetilde{x}^i) \varphi_k(\widetilde{x}^i) e^j w(\widetilde{x}^j) \varphi_k(\widetilde{x}^j) \vert K ) = \sigma^2 \mathbb{E}( w(\widetilde{x}^i) \varphi_k(\widetilde{x}^i)^2 \vert K ).$$
		Then 
		\begin{equation}
		\begin{aligned}
		\sum_{k=1}^m	\mathbb{E}(\vert \beta_k \vert^2) =	\sigma^2  \mathbb{E}\left( \frac{1}{\#K^2} \sum_{i \in K}  w(\widetilde{x}^i)^2 \sum_{k=1}^m \varphi_k(\widetilde{x}^i)^2\right)  & = m \sigma^2 \mathbb{E}\left( \frac{1}{\#K^2} \sum_{i \in K}  w(\widetilde{x}^i)^2 \right)\\
		& \le \frac{m}{n_{min}^2}\sigma^2 \mathbb{E}\left(\sum_{i \in K} w(\widetilde{x}^i)\right)\\
		& \le \frac{m}{n_{min}^2}\sigma^2 \mathbb{E}\left(\sum_{i =1}^n w(\widetilde{x}^i)\right).\\
		\end{aligned}
		\end{equation}
		To bound the term, $\mathbb{E}\left( \sum_{i=1}^n w(\widetilde{x}^i) \right)$, we use \eqref{norm-tildexn-bound} with $v=1$,
		\begin{equation}
		\begin{aligned}
		\mathbb{E}\left( \sum_{i=1}^n w(\widetilde{x}^i) \right) 	&  \le n\mathbb{E}\left( \Vert 1 \Vert_{\widetilde{x}^n}^2  \right) \le nM (1-\eta^M)^{-1} \Vert 1 \Vert^2 = nM (1-\eta^M)^{-1}. \\
		\end{aligned}
		\end{equation}
		All in all,
		\begin{equation}
		\begin{aligned}
		\sum_{k=1}^m	\mathbb{E}(\vert \beta_k \vert^2) \le \sigma^2\frac{mn}{n_{min}^2} (1-\eta^M)^{-1}M.\\
		\end{aligned}
		\end{equation}
	\end{proof}
	When there is no subsampling the bound from Equation \ref{quasi-optim-greedy-noisy} becomes
	\begin{equation}\label{quasi-optim-greedy-noisy-no-subs}\small
	\mathbb{E}(\Vert u - u^{\widetilde{\bm{x}}^n_K} \Vert^2) \le (1+  (1-\delta)^{-1}(1 - \eta^M)^{-1} M)\Vert u - P_{V_m}u \Vert^2 + \frac{2\sigma^2m}{n}(1-\delta)^{-1}(1-\eta^M)^{-1}M.
	\end{equation}
	When $n = n_{min}$ (allowing subsampling to reach interpolation regime $\#K =m$), the bound becomes
	\begin{equation}\label{quasi-optim-greedy-noisy-full-subs}\small
	\mathbb{E}(\Vert u - u^{\widetilde{\bm{x}}^n_K} \Vert^2) \le (1+  \frac{2n}{m}(1-\delta)^{-1}(1 - \eta^M)^{-1} M)\Vert u - P_{V_m}u \Vert^2 + 2\sigma^2\frac{n}{m}(1-\delta)^{-1}(1-\eta^M)^{-1}M ,
	\end{equation}
	in this particular case the influence of the noise may be more important, as for an interpolation method. Then in the noisy case, using $n_{min} = \frac{n}{\beta}$ for some fixed $\beta > 1$ allows to better control the noise term.
	\section{Numerical experiments}\label{sec:examples}
	\subsection{Notations and objectives}
	\label{subs:notations}
	In this section, we focus on polynomial approximation spaces $V_m = \mathbb{P}_p$ with $p$ the polynomial degree. We use an orthonormal polynomial basis of $V_m$ (Hermite polynomials for a Gaussian measure or Legendre polynomials for a uniform measure). The aim is to compare the performance of the method we propose with the optimal weighted least-squares method and interpolation. We are not trying to be exhaustive in this comparison, but to cover a quite large panel of state-of-the art approximation methods.\\ First, we consider interpolation performed on deterministic set of points (Gauss-Hermite points for a Gaussian measure, abbreviated \textbf{$\bm{\mathcal{I}}$-GaussH} and Gauss-Legendre points for a uniform measure, abbreviated \textbf{$\bm{\mathcal{I}}$-GaussL}), magic points, abbreviated \textbf{$\bm{\mathcal{I}}$-Magic}, see \cite{Maday2007}, Leja points, abbreviated \textbf{$\bm{\mathcal{I}}$-Leja}, see \cite{Baglama1998},\cite{Bos2010} and \cite{Narayan2014} for their weighted version dealing with unbounded domains and Fekete points, abbreviated \textbf{$\bm{\mathcal{I}}$-Fekete}, see \cite{Sommariva2009} and \cite{Guo2018} for their weighted version dealing with unbounded domains. The three last sets of points are chosen among a sufficiently large and dense discretization of $\mathcal{X}$.\\ Then, we consider least-squares methods, more precisely standard least-squares methods, abbreviated \textbf{SLS}, optimal weighted least-squares projection (introduced in \cite{Cohen2016}), abbreviated \textbf{OWLS}, and also the boosted optimal weighted least-squares projections we propose, abbreviated \textbf{BLS}, \textbf{c-BLS} and \textbf{s-BLS} when we respectively use resampling, conditioning, and subsampling plus resampling and conditioning.
	\begin{algorithm}[H]
		\caption{Presentation of the \textbf{BLS} method, \textbf{c-BLS} method and \textbf{s-BLS} method}\label{algo:bls_algo}
		\begin{algorithmic}
			\STATE \hspace{-0.4cm}\textbf{Inputs:} $\delta$, $\eta$, $M$, $V_m$. 
			\STATE \hspace{-0.4cm}\textbf{Outputs:} $\bm{x}^{n,\star}$ for the \textbf{BLS} method, $\widetilde{\bm{x}}^{n}$ for the \textbf{c-BLS} method, $\widetilde{\bm{x}}^{n}_K$, for the \textbf{s-BLS} method.
			\STATE \
			\FOR{$i=1, \hdots, M$}
			\STATE Sample $\bm{x}^{n,i} \sim \bm{\rho}^n$
			\ENDFOR
			\STATE $\bm{x}^{n,\star} = \min_{1 \le i \le M}\Vert \bm{G}_{\bm{x}^{n,i}} - \bm{I} \Vert_2$.
			\STATE $\bm{x}^{n,\star}$ is the sample produced with the \textbf{BLS} method.
			\STATE \ 
			\STATE Initialize : $\widetilde{\bm{x}}^{n}$ is a sample produced with the \textbf{BLS} method.
			\WHILE{$\Vert \bm{G}_{\bm{x}^{n,\star}} - \bm{I} \Vert_2 > \delta$}
			\STATE Sample $\widetilde{\bm{x}}^{n}$ with the \textbf{BLS} method.
			\ENDWHILE
			\STATE $\widetilde{\bm{x}}^{n}$ is the sample produced with the \textbf{c-BLS} method.
			\STATE \ 
			\STATE Initialize $K = \{1, \hdots, n\}$, $\widetilde{\bm{x}}^{n}_K$ is a sample produced with the \textbf{c-BLS} method.
			\WHILE{$\Vert \bm{G}_{\widetilde{\bm{x}}^{n}_K} - \bm{I} \Vert_2 < \delta$}
			\STATE Select $k^{\star}$ such that $$
			\Vert \bm{G}_{{\widetilde{\bm{x}}^n}_{K \setminus \{k^{\star}\}}} - \bm{I} \Vert_2 = \min_{k \in K} \Vert \bm{G}_{{\widetilde{\bm{x}}^n}_{K \setminus \{k\}}} - \bm{I} \Vert_2.
			$$
			\STATE $K = K \setminus k^{\star}$
			\ENDWHILE
			\STATE $\widetilde{\bm{x}}^{n}_K$ is the sample produced with the \textbf{s-BLS} method.
		\end{algorithmic}
	\end{algorithm}
	\begin{remark}
		For a fixed approximation space $V_m$, it must be noticed that the methods \textbf{OWLS}, \textbf{BLS} and \textbf{$\bm{\mathcal{I}}$-GaussH}, \textbf{$\bm{\mathcal{I}}$-GaussL} or \textbf{$\bm{\mathcal{I}}$-Leja} points, do not depend on the choice of the orthonormal basis associated with $V_m$, as the quantity $Z_{\bm{x}^n}$ is independent of this choice. This is however not the case for the two following methods \textbf{$\bm{\mathcal{I}}$-Magic} \cite{Maday2007} or \textbf{$\bm{\mathcal{I}}$-Fekete} \cite{Sommariva2009}, the particular choice of the basis will be mentioned in each example.
	\end{remark}
	\begin{remark}	
		The difficulty of the optimal least-squares methods lies in the fact that we sample from a non-usual probability density function. We need a sampling technique that is both accurate and as fast as possible.\\
		Sampling from univariate densities can be done with classical techniques as rejection sampling \cite{Devroye1985}, inverse transform sampling \cite{Devroye1985} and slice sampling \cite{Neal2003}. This latter is used for all the numerical examples.
		To sample from multivariate densities we use a sequential sampling technique, as described in \cite{Cohen2016}, which only requires samplings from univariate densities.
	\end{remark}
	In the next section, two kinds of comparisons are performed. 
	First, we compare qualitatively the distributions of the random variable $Z_{\bm{x}^n}$ and the distributions of the $n$-points sample $\bm{x}^n$. These analyses depend only on the choice of the approximation space $V_m$, and does not involve a function to approximate.
	Secondly, we compare quantitatively the efficiency of the different methods to approximate functions. We consider analytical functions on $\mathbb{R}^d$ or $[-1,1]^d$ equipped with Gaussian or uniform measures.\\
	
	\subsection{Qualitative analysis of the boosted optimal weighted least-squares method}
	\subsubsection{Analysis of the stability}
	The objective of this paragraph is to compare the stability of the boosted optimal weighted least-squares method, using subsampling from Section \ref{sec:subsampling}  or not, respectively \textbf{s-BLS} and \textbf{c-BLS}, with two other state-of-the art methods, standard least-squares method, abbreviated \textbf{SLS} and \textbf{OWLS} method. As explained in Section \ref{sec:wls}, the stability of the least-squares projection can be characterized by the random variable $Z_{\bm{x}^n} = \Vert  \bm{G}_{\bm{x}^n} - \bm{I} \Vert_2$. The closer $Z_{\bm{x}^n}$ is to 1, the more stable the approximation is. In this paragraph, we compare the distribution of this random variable $Z_{\bm{x}^n}$ for the different sampling methods. For the \textbf{SLS} method, the sampling measure is the reference measure $\mu$. In the \textbf{OWLS} method, the sampling measure $\rho$ is the measure with density $w^{-1}$ with respect to the reference measure $\mu$,  chosen as in \eqref{optimal-w}. \\
	We present results for approximation spaces $V_m = \mathbb{P}_5$ with $\mu$ a Gaussian or uniform measure. The Figures \ref{fig:distribGL1} and \ref{fig:distribGH1} show that using \textbf{OWLS} instead of \textbf{SLS} shifts to the left the distribution of the random variable $Z_{\bm{x}^n}$. Without surprise, we see that conditioning $Z_{\bm{x}^n}$ by the event $A_{\delta} = \{ Z_{\bm{x}^n} \le \delta \}$ yields a distribution whose support is included in $[0,\delta]$. As expected, we also notice that very similar results are obtained for the \textbf{OWLS} and \textbf{c-BLS} methods when choosing $M=1$. In the same manner, increasing the number of resampling $M$ also shifts the PDF of $Z_{\bm{x}^n}$ to the low values, and decreases its variability. When interested in maximizing the probability of $A_{\delta}$, \textbf{c-BLS} method is therefore an interesting alternative to \textbf{SLS} and \textbf{OWLS}. At last,  looking at Figures \ref{fig:distribGL2} and \ref{fig:distribGH2}, we observe that the greedy selection moves the PDF of $Z_{\bm{x}^n}$ to the high values. This was expected: to switch from \textbf{c-BLS} to \textbf{s-BLS}, the size of the sample is reduced, as points are adaptively removed. However, as it is conditioned by $A_{\delta}$, it remains better than \textbf{SLS} and \textbf{OWLS} methods.\\
	
	\begin{center}
		\begin{figure}
			\begin{subfigure}[b]{0.49\textwidth}
				\centering
				\includegraphics[scale =1]{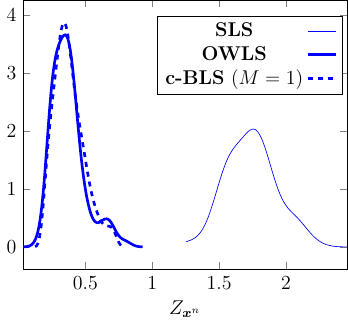}
				\includegraphics[scale =1]{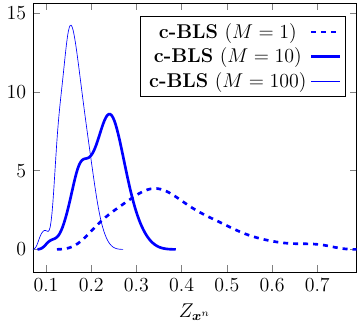}
				\caption{Gaussian measure}
				\label{fig:distribGH1}	
			\end{subfigure}
			\begin{subfigure}[b]{0.49\textwidth}
				\centering
				\includegraphics[scale =1]{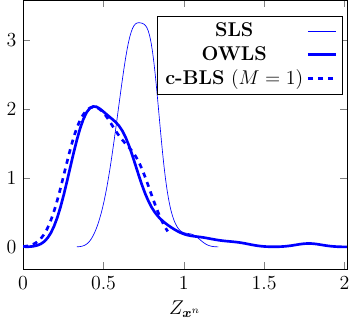}
				\includegraphics[scale =1]{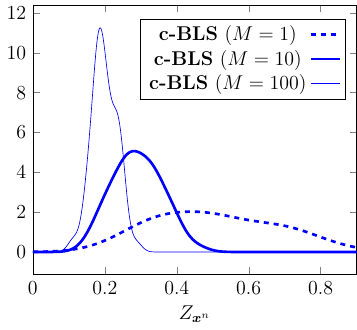}
				\caption{Uniform measure}
				\label{fig:distribGL1}	
			\end{subfigure}
			\caption{Probability density function of $ Z_{\bm{x}^n} = \Vert \bm{G}_{\bm{x}^n} -\bm{I} \Vert_2 $ for $V_m =\mathbb{P}_5$, with $\delta = 0.9$ and $n = 100$.}
		\end{figure}
	\end{center}
	\begin{figure}
		\begin{subfigure}[b]{0.49\textwidth}
			\centering
			\includegraphics[scale =1]{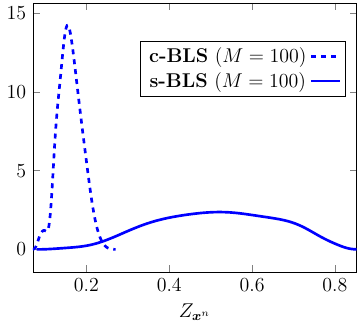}
			\caption{Gaussian measure}
			\label{fig:distribGH2}	
		\end{subfigure}
		\begin{subfigure}[b]{0.49\textwidth}
			\centering
			\includegraphics[scale =1]{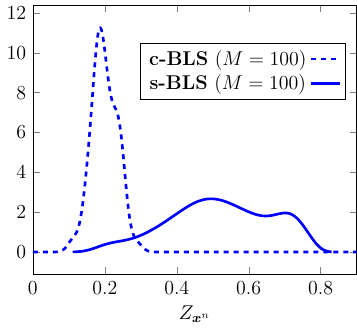}
			\caption{Uniform measure}
			\label{fig:distribGL2}	
		\end{subfigure}
		\caption{Probability density function of $ Z_{\bm{x}^n} = \Vert \bm{G}_{\bm{x}^n} -\bm{I} \Vert_2 $ for $V_m =\mathbb{P}_5$, with $\delta = 0.9$ and $n=100$.}
	\end{figure}
	\subsubsection{Distribution of the sample points} 
	In this paragraph, we are interested in the distributions of the points sampled with the \textbf{c-BLS} and \textbf{s-BLS} methods. We consider $d=1$.\\ 
	First, $n=10$ points are sampled according the \textbf{c-BLS} method for different values of $M$ (from $1$ to $50000$). These points are then sorted in ascending order. After repeating this procedure $r=1000$ times, the probability distributions of the sorted points are represented in Figures \ref{fig:repartition_pointsHnG} and \ref{fig:repartition_pointsLnG} (one color per point).\\
	\begin{figure}
		\begin{subfigure}[b]{0.49\textwidth}
			\centering
			\includegraphics[scale =1]{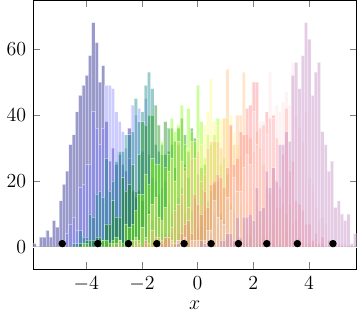}
			\caption{$M=1$}
		\end{subfigure}
		\begin{subfigure}[b]{0.49\textwidth}
			\centering
			\includegraphics[scale =1]{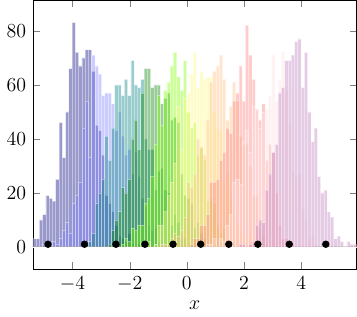}
			\caption{$M=10$}
		\end{subfigure}
		\begin{subfigure}[b]{0.49\textwidth}
			\centering
			\includegraphics[scale =1]{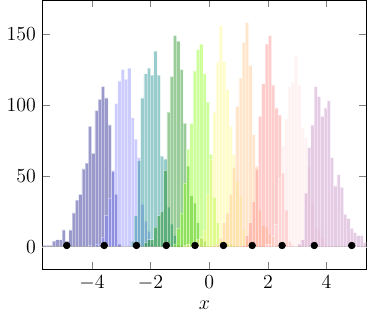}
			\caption{$M=10000$}
		\end{subfigure}
		\begin{subfigure}[b]{0.49\textwidth}
			\centering
			\includegraphics[scale =1]{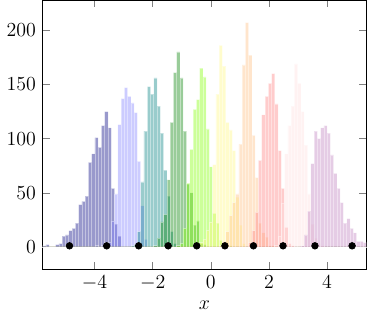}
			\caption{$M=50000$}
			\label{fig:repartition_pointsHnGM50000}
		\end{subfigure}
		\caption{Distributions of the $x^{(i)}, i=1, \hdots, 10$, with $\bm{x}^{10}$ sampled from the \textbf{c-BLS} method for $V_m = \mathbb{P}_5$ and $\mu$ the Gaussian measure.}
		\label{fig:repartition_pointsHnG}
	\end{figure}
	For $\mu$  the Gaussian or the uniform measure, when $M$ is small, ($M=1$ or $M=10$), we notice a strong overlap between the support of the different distributions. This is no longer the case for the highest values of $M$ ($M=10000$ or $M=50000$). Hence, the larger $M$, the further apart the points are from each other with high probability, and the more they concentrate around specific values.
	\begin{figure}
		\begin{subfigure}[b]{0.49\textwidth}
			\centering
			\includegraphics[scale =1]{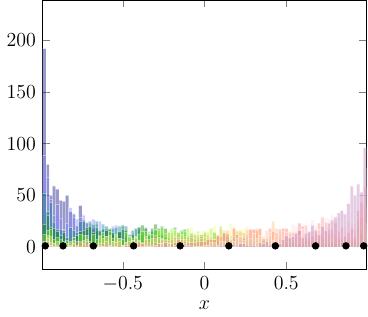}
			\caption{$M=1$}
		\end{subfigure}
		\begin{subfigure}[b]{0.49\textwidth}
			\centering
			\includegraphics[scale =1]{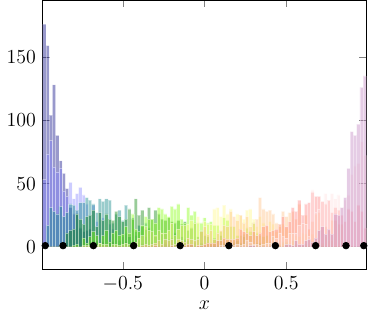}
			\caption{$M=10$}
		\end{subfigure}
		\begin{subfigure}[b]{0.49\textwidth}
			\centering
			\includegraphics[scale =1]{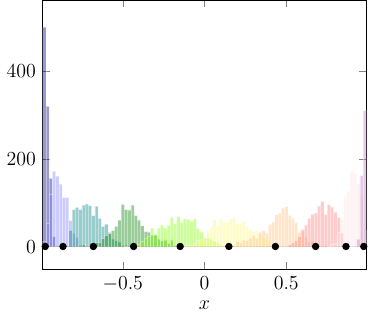}
			\caption{$M=10000$}
		\end{subfigure}
		\begin{subfigure}[b]{0.49\textwidth}
			\centering
			\includegraphics[scale =1]{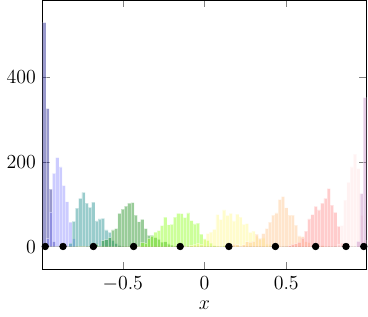}
			\caption{$M=50000$}
			\label{fig:repartition_pointsLnGM50000}
		\end{subfigure}
		\caption{Distributions of the $x^{(i)}, i=1, \hdots, 10$, with $\bm{x}^{10}$ sampled from the \textbf{c-BLS} method for $V_m = \mathbb{P}_5$ and $\mu$ the uniform measure.}
		\label{fig:repartition_pointsLnG}
	\end{figure}
	Secondly, $n = 6$ points are sampled according to the \textbf{s-BLS} method. To this end, a greedy procedure is applied to remove points from an initial sample of $10$ points until we get the required number of points. In that case, as we fix the size of the sample, there is a priori no guarantee that the value of $Z_{\bm{x}^n}$ remains smaller than $\delta$.
	The obtained $6$-points sample is once again sorted in ascending order, and we repeat the procedure $r=1000$ times. As previously, the distributions of the sorted points are represented in Figures \ref{fig:repartition_pointsHG} and \ref{fig:repartition_pointsLG} for different values of $M$. Only moderate values of $M$ are considered, as we empirically observed that choosing $M$ higher than $100$ had very little influence on the results when considering subsampling. \\
	\begin{figure}[tbhp]
		\begin{subfigure}[b]{0.49\textwidth}
			\centering
			\includegraphics[scale =1]{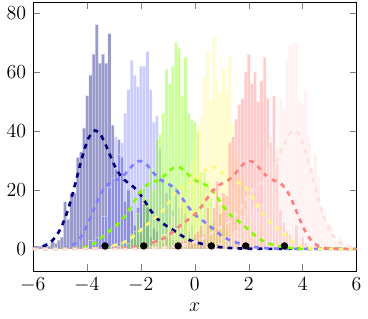}
			\caption{$M=1$}
			\label{fig:repartition_pointsHGM1}
		\end{subfigure}
		\begin{subfigure}[b]{0.49\textwidth}
			\centering
			\includegraphics[scale =1]{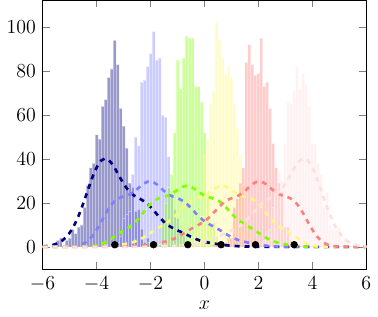}
			\caption{$M=10$}
		\end{subfigure}
		\begin{subfigure}[b]{0.49\textwidth}
			\centering
			\includegraphics[scale =1]{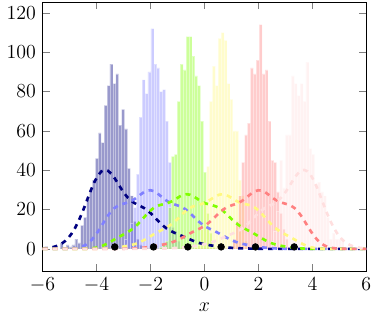}
			\caption{$M=50$}
		\end{subfigure}
		\begin{subfigure}[b]{0.49\textwidth}
			\centering
			\includegraphics[scale =1]{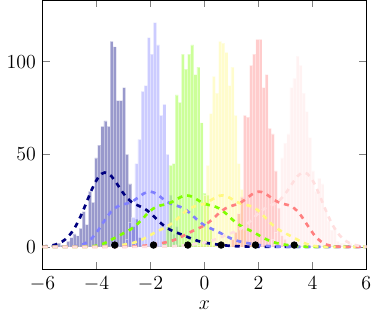}
			\caption{$M=100$}
			\label{fig:repartition_pointsHGM100}
		\end{subfigure}
		\caption{Distributions of the $x^{(i)}, i=1, \hdots, 6$, with $\bm{x}^{6}$ sampled from the \textbf{s-BLS} method (colored histograms) and  \textbf{OWLS} method (dashed lines) for $V_m = \mathbb{P}_5$ with $\mu$ the gaussian measure.}
		\label{fig:repartition_pointsHG}
	\end{figure} 
	Comparing the figures associated with the methods with or without greedy subsampling, we finally observe that \textbf{s-BLS} method provides results that are very close to \textbf{c-BLS} method with a very high value of $M$. This emphasizes the efficiency of the greedy selection to separate the support of the distributions of points.\\
	In Figure \ref{fig:repartition_pointsHGM1}, the distributions of the sorted points associated to the \textbf{OWLS} method are represented in dashed lines. This shows that even if no resampling is carried out ($M=1$), using the \textbf{s-BLS} method instead of \textbf{OWLS} improves the space filling properties of the obtained samples.
	\begin{remark}
		In Figures \ref{fig:repartition_pointsHnG}, \ref{fig:repartition_pointsLnG}, \ref{fig:repartition_pointsHG} and \ref{fig:repartition_pointsLG} black dots have been added to indicate the positions of the first $n$ Gauss-Hermite points in the Gaussian case, and the $n$ first Gauss-Legendre points in the uniform case. Interestingly, we observe that, in the Gaussian case, the distribution of points spreads symmetrically around zero and in the uniform case, the distribution concentrates on the edges.
	\end{remark}
	\begin{figure}
		\begin{subfigure}[b]{0.49\textwidth}
			\centering
			\includegraphics[scale =1]{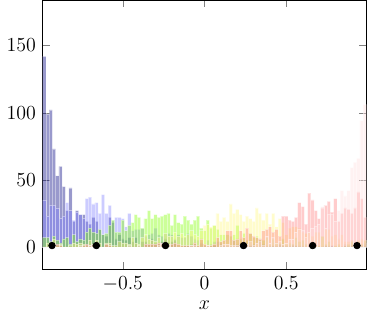}
			\caption{$M=1$}
		\end{subfigure}
		\begin{subfigure}[b]{0.49\textwidth}
			\centering
			\includegraphics[scale =1]{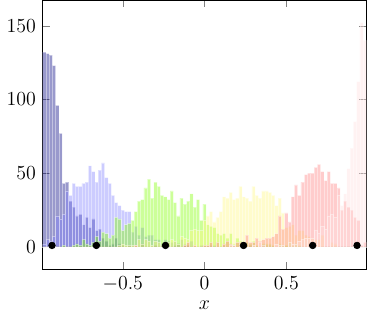}
			\caption{$M=10$}
		\end{subfigure}
		\begin{subfigure}[b]{0.49\textwidth}
			\centering
			\includegraphics[scale =1]{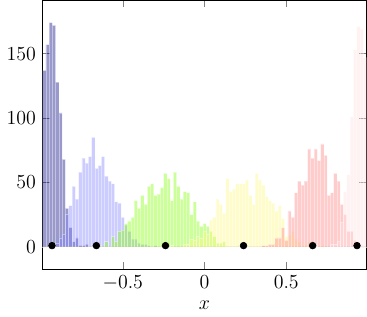}
			\caption{$M=50$}
		\end{subfigure}
		\begin{subfigure}[b]{0.49\textwidth}
			\centering
			\includegraphics[scale =1]{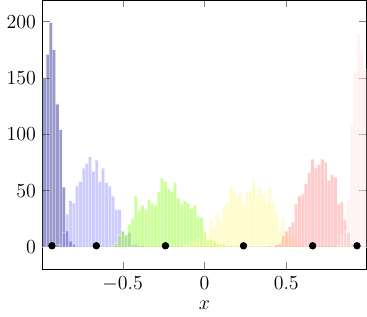}
			\caption{$M=100$}
			\label{fig:repartition_pointsLGM100}
		\end{subfigure}
		\caption{Distributions of the $x^{(i)}, i=1, \hdots, 6$, with $\bm{x}^{6}$ sampled from the \textbf{s-BLS} method for $V_m = \mathbb{P}_5$ and $\mu$ the uniform measure. }
		\label{fig:repartition_pointsLG}
	\end{figure}  
	\subsection{Quantitative analysis for polynomial approximations}
	\label{subs:examples_functions}
	In this paragraph, we want to compare the different methods introduced in Subsection \ref{subs:notations} in terms of approximation efficiency. The quality of the approximation $u^{\star}$ of a function $u \in L^2_{\mu}(\mathcal{X})$ is assessed by estimating the error of approximation 
	$$ \varepsilon^2 = \frac{1}{N_{test}}\sum_{i \in \bm{x}_{test}} (u(x^{(i)})  - u^{\star}(x^{(i)}))^2.$$ In practice, we choose $N_{test} = 1000$.
	To study the robustness of the methods, we compute $10$ times the approximations and draw $10$ different test samples $\bm{x}_{test}$ and compute empirical confidence intervals of level $10\%$ and $90\%$ for the errors of approximation.\\
	
	For each example, two kind of comparisons are performed.
	\begin{itemize}
		\item Tables (a) present the results for the methods \textbf{OWLS}, \textbf{c-BLS}, and  \textbf{s-BLS}, the number of samples is chosen to ensure the stability of the empirical Gram matrix with high probability or almost surely. For the \textbf{OWLS} method, the number of samples is equal to $n = \lceil d_{\delta}^{-1}m\log(2m\eta^{-1}) \rceil $, with $\delta = 0.9$ and $\eta = 0.01$ such that the stability of the empirical Gram matrix is ensured with probability greater than 0.99. For the \textbf{c-BLS} method, the initial number of samples $n$ is also equal to $n = \lceil d_{\delta}^{-1}m\log(2m\eta^{-1}) \rceil $ but choosing $\delta = 0.9$ and $\eta = 0.01^{1/M}$ ($M$ is specified in the example). Since the resulting sample $\widetilde{\bm{x}}^n$ satisfies $A_{\delta}$, the stability of the empirical Gram matrix is guaranteed. For the \textbf{s-BLS} method, the initial sample is the same than for the \textbf{c-BLS} method and a  greedy removal of points is performed as long as the event $A_{\delta}$ is satisfied. 
		\item Tables (b) compare all the $\bm{\mathcal{I}-}$ methods with \textbf{OWLS}, \textbf{BLS} and \textbf{s-BLS}. For all methods, except \textbf{s-BLS}, the number of samples $n$ is taken equal to the dimension of the approximation space $m$. In this particular comparison, the \textbf{BLS} method only consists in a resampling strategy. For the \textbf{s-BLS} method, the initial number of samples is equal to $n = \lceil d_{\delta}^{-1}m\log(2m\eta^{-1}) \rceil $, with $\delta = 0.9$ and $\eta = 0.01^{1/M}$ ($M$ is specified in the example) and since the resulting sample $\widetilde{\bm{x}}^n$ satisfies $A_{\delta}$, the stability of the empirical Gram matrix is guaranteed. Then the greedy selection of points is performed as long as $n \ge m$, implying that the resulting sample used to build the approximation may not lead to a stable empirical Gram matrix.
		\begin{remark}
			In the interpolation regime $n=m$, the stability condition from \ref{bound-PZ} can not be reached. Indeed, choosing $M$ arbitrary big enables us to choose $\eta$ close to $1$, but still $\eta < 1$. It implies that the number of samples $n(\delta, \eta, m)$ necessary to get the stability condition from Theorem \ref{th:condition-sample-size} has to be greater than $d_{\delta}^{-1}m\log(2m) >m$. In the case of controlled cost, this explains why we choose to use the \textbf{BLS} method without conditioning. 
		\end{remark}
		\label{s-BLSnbofsamples}
	\end{itemize}
	% To ensure the stability of the projection, the \textbf{OWLS} and \textbf{BLS} methods require a minimal number of samples given by the condition from Eq \eqref{minimumsample}. The efficiency of the different methods are finally compared for classical (Hermite or Legendre polynomials essentially) and non-classical approximation spaces.\\
	% The interest of using a non-classical basis is to show the limitations of deterministic interpolation \textbf{I}. Indeed, in that case, there might not exist adapted sequences to construct the interpolation operator, which can make the interpolation unstable when $m$ increases. 
	\subsubsection{A first function}
	\label{subsubs:first_function}
	We consider $\mathcal{X} = \mathbb{R}$, equipped with the standard Gaussian measure $\mu$ and the function
	\begin{equation}
	u_1(x) = \exp\left(-\frac{1}{4}(x-1)^2\right).
	\end{equation}
	The approximation space is $V_m = \mathbb{P}_{m-1} = \text{span}\{\varphi_i: 1 \le i \le m\}$, where the basis $\{\varphi_i\}_{i=1}^m$ is chosen as the Hermite polynomials of degree less than $m-1$. This is referred to as example 1.
	\begin{table}
		\small
		\centering
		\begin{subtable}{.99\textwidth}
			\centering
			\begin{tabular}{cc|c||c|c||c|c||c|c}
				
				& \multicolumn{2}{c||}{\textbf{OWLS}} &\multicolumn{2}{c||}{\makecell{{\textbf{c-BLS}} \\ $(M=1)$}}& \multicolumn{2}{c||}{\makecell{{\textbf{c-BLS}} \\ $(M=100)$}} & \multicolumn{2}{c}{\makecell{{\textbf{s-BLS}} \\$(M=100)$}}\\                                                                                                                     
				\rowcolor{gray!10!white}  $p$ & $\varepsilon$ & $n$ & $\varepsilon$ & $n$ & $\varepsilon$ & $n$ & $\varepsilon$ & $n$ \\                                           
				5 & [-2.1; -1.8] & 134 & [-2.2; -1.9] & 134 & [-2.2; -1.8] & 48 & [-2.1; -1.8] & [6; 6] \\       
				
				\rowcolor{gray!10!white}  10 & [-3.1; -3.1] & 265 & [-3.2; -3.0] & 265 & [-3.2; -3.1] & 108 & [-3.0; -2.6] & [11; 13] \\ 
				15 & [-4.5; -4.2] & 404 & [-4.5; -4.3] & 404 & [-4.5; -4.3] & 176 & [-4.6; -4] & [16; 17] \\                                                                                                                                           
				\rowcolor{gray!10!white}  20 & [-5.8; -5.7] & 548 & [-5.9; -5.7] & 548 & [-5.9; -5.7] & 249 & [-5.9; -5.3] & [21; 23] \\
				25 & [-6.9; -6.7] & 697 & [-7; -6.8] & 697 & [-7; -6.7] & 326 & [-6.9; -6.6] & [26; 28] \\      
				
				\rowcolor{gray!10!white}  30 & [-8.3; -8.1] & 848 & [-8.4; -8.2] & 848 & [-8.4; -8.2] & 405 & [-8.4; -8.0] & [31; 36] \\           
				35 & [-9.4; -9.3] & 1001 & [-9.4; -9.3] & 1001 & [-9.4; -9.2] & 488 & [-9.4; -8.9] & [37; 39] \\                                                                                                                                                    
				\rowcolor{gray!10!white}  40 & [-10.7; -10.5] & 1157 & [-10.7; -10.5] & 1157 & [-10.7; -10.6] & 572 & [-10.8; -10.3] & [42; 44] \\                                                                                                                                                         
			\end{tabular}                                                                                                                                                             
			\caption{Guaranteed stability with probability greater than $0.99$ for \textbf{OWLS} and almost surely for the other methods.}
			\label{table:ex1Hprecision}
		\end{subtable}
		\begin{subtable}{.99\textwidth}
			\centering
			\begin{tabular}{cc|c|c|c||c|c|c}
				\rowcolor{gray!10!white}  & \multicolumn{4}{c||}{\textbf{Interpolation}} & \multicolumn{3}{c}{\textbf{Least-Squares regression}}\\
				& \multicolumn{1}{c|}{$\bm{\mathcal{I}}$\textbf{-GaussH}} & \multicolumn{1}{c|}{$\bm{\mathcal{I}}$\textbf{-Magic}} & \multicolumn{1}{c|}{$\bm{\mathcal{I}}$\textbf{-Leja}} & \multicolumn{1}{c||}{$\bm{\mathcal{I}}$\textbf{-Fekete}} & \multicolumn{1}{c|}{\textbf{OWLS}} & \multicolumn{1}{c|}{\makecell{{\textbf{BLS}} \\ $(M=100)$}} & \multicolumn{1}{c}{\textbf{s-BLS}}\\
				\rowcolor{gray!10!white} $p$ & $\varepsilon$ & $\varepsilon$ & $\varepsilon$ & $\varepsilon$ & $\varepsilon$ & $\varepsilon$ & $\varepsilon$\\                                                                                                                                                                                                                                                                                                                                                                                                 
				5 &  [-2.2,-1.5] & [-0.5; -0.5] & [-1.3; -1.2] & [-1.3; -1.2] & [-1.8; -0.3] & [-2.1; -1.6] & [-2.1; -1.8]  \\                                                                                                                                      
				
				\rowcolor{gray!10!white}  10 & [-3.1,-3.0]   & [-0.8; -0.8] & [-2.1; -2.0] & [-1.9; -1.9] & [-2.6; -1.0] & [-3.1; -1.9] & [-3.0; -2.5]  \\                                                                                                                                   
				15 &  [-4.5,-4.2] & [-1.4; -1.4] & [-2.6; -2.5] & [-2.4; -2.3] & [-4.4; -1.6] & [-4.3; -2.2] & [-4.6; -3.9]  \\                                                                                                                                   
				
				\rowcolor{gray!10!white}  20 & [-5.9,-5.8]   & [-2.2; -2.2] & [-4.2; -4.2] & [-4.1; -4.1] & [-5.9; -3.2] & [-6.1; -4.3] & [-5.9; -5.1]  \\                                                                                                                                  
				25 &  [-6.8,-6.7]  & [-3.4; -3.4] & [-4.4; -4.3] & [-5.1; -5.0] & [-6.8; -3.9] & [-6.4; -4.7] & [-6.9; -6.0]  \\                                                                                                                                 
				
				\rowcolor{gray!10!white}  30 & [-8.6,-8.4]  & [-4.5; -4.4] & [-5.7; -5.6] & [-6.1; -6.1] & [-9.1; -4.4] & [-8.8; -5.2] & [-8.5; -7.2]  \\                                                                                                                                   
				35 &  [-9.3,-9.2] & [-5.3; -5.3] & [-7.1; -7.0] & [-7.7; -7.6] & [-10.4; -4.9] & [-9.0; -6.4] & [-9.4; -8.5]  \\                                                                                                                                   
				
				\rowcolor{gray!10!white}  40 & [-9.8,-9.7] & [-7.0; -7.0] & [-8.6; -8.5] & [-8.5; -8.5] & [-10.9; -6.7] & [-10.8; -7.1] & [-11.0; -9.7]  \\                                                                                                                      
				
			\end{tabular}                                                                                                                                                                                                                                                                                                                                                       
			\caption{Given cost: $n= m$}
			\label{table:ex1Hcost}
		\end{subtable}
		\caption{Approximation error $\varepsilon$ in log-10 scale for the example 1. Abbreviations are defined in Subsection \ref{subs:notations}.}
		\label{table:ex1H}
	\end{table}
	
	For this example, looking at Table \ref{table:ex1Hprecision}, we first observe that the approximation error decreases in a similar way for the four methods \textbf{OWLS}, \textbf{c-BLS} ($M=1$), \textbf{c-BLS} ($M=100$) and \textbf{s-BLS} ($M=100$) when the size of the approximation space increases. However, the results for the \textbf{c-BLS} (M=100) method are using less evaluations of the function. Indeed, by resampling, that is to say by increasing the value of $M$, the bound of the probability of getting a stable approximation is $1-\eta^M$ instead of $1-\eta$. Hence if $\eta$ is chosen equal to 0.01 for $M=1$, taking $\eta$ equal to $0.01^{1/M}$ for higher values of $M$ does not modify the bound of the probability of getting a stable approximation, but allows us to strongly reduce the number of samples needed to guarantee the same stability condition (see \ref{numberofsamples} for the explicit relation between the minimum number of samples and $\eta$). Regarding the number of evaluations of the function, the \textbf{s-BLS} ($M=100$) method, which uses greedy subsampling is even better.\\
	
	Looking at Table \ref{table:ex1Hcost}, we also observe that for all the methods, the error of approximation decreases when the size of the approximation space increases. Nevertheless, it is interesting to notice that among the interpolation methods, the $\bm{\mathcal{I}}$\textbf{-Magic} method is less accurate than the others when the dimension of the approximation space is too high. In practice, for the \textbf{s-BLS} method, letting the greedy algorithm reach the interpolation regime $(m=n)$ may provide a sample $\widetilde{\bm{x}}^n$ which does not guarantee the stability. However, focusing on the upper bound of the errors, we also see that in the interpolation regime, only the \textbf{s-BLS} method seems to be able to provide results that are compatible to the ones of the $\bm{\mathcal{I}}$\textbf{-GaussH} method.
	\begin{remark}
		The points for the $\bm{\mathcal{I}}$\textbf{-Magic}, $\bm{\mathcal{I}}$\textbf{-Fekete} and $\bm{\mathcal{I}}$\textbf{-Leja} are chosen among a sufficiently large and dense discretization of $\mathcal{X}$. Here, in the case where $\mathcal{X} = \mathbb{R}$, we choose a uniform discretization of the bounded interval $[-10, 10]$ with $10000$ points.
	\end{remark}
	\subsubsection{A second function}
	In this section, we consider $\mathcal{X} = [-1, 1]$ equipped with the uniform measure and the function
	\begin{equation}
	u_2(x) = \frac{1}{1+5x^2}.
	\end{equation}		
	The approximation space is $V_m = \mathbb{P}_{m-1} = \text{span}\{\varphi_i: 1 \le i \le m\}$, where the basis $\{\varphi_i\}_{i=1}^m$ is chosen as the Legendre polynomials of degree less than $m-1$. This is referred to as example 2.
	\begin{table}
		\small
		\begin{subtable}{.99\textwidth}
			\centering
			\begin{tabular}{cc|c||c|c||c|c||c|c}   
				& \multicolumn{2}{c||}{\textbf{OWLS}} & \multicolumn{2}{c||}{\makecell{{\textbf{c-BLS}} \\ $(M=1)$}} &  \multicolumn{2}{c||}{\makecell{{\textbf{c-BLS}} \\ $(M=100)$}} &  \multicolumn{2}{c}{\makecell{{\textbf{s-BLS}} \\ $(M=100)$}}\\
				\rowcolor{gray!10!white}  $p$ & $\varepsilon$ & $n$ & $\varepsilon$ & $n$ & $\varepsilon$ & $n$ & $\varepsilon$ & $n$ \\                                   
				5 & [-1.3; -1.2] & 134 & [-1.3; -1.2] & 134 & [-1.3; -1.3] & 48 & [-1.2; -0.9] & [6; 6] \\         
				
				\rowcolor{gray!10!white}  10 & [-2.4; -2.4] & 265 & [-2.4; -2.4] & 265 & [-2.4; -2.4] & 108 & [-2.3; -1.9] & [11; 11] \\      
				
				15 & [-3.1; -3.1] & 405 & [-3.2; -3.2] & 405 & [-3.2; -3.2] & 176 & [-3.1; -2.8] & [16; 16] \\     
				
				\rowcolor{gray!10!white}  20 & [-4.3; -4.2] & 548 & [-4.3; -4.3] & 548 & [-4.3; -4.3] & 249 & [-4.2; -4.1] & [21; 23] \\    
				25 & [-5.0; -4.8] & 697 & [-5.1; -5.0] & 697 & [-5.1; -5.0] & 326 & [-5.0; -4.7] & [26; 29] \\        
				
				\rowcolor{gray!10!white}  30 & [-6.2; -6.1] & 848 & [-6.2; -6.2] & 848 & [-6.2; -6.2] & 405 & [-6.1; -5.8] & [31; 35] \\      
				35 & [-6.9; -6.9] & 1001 & [-6.9; -6.9] & 1001 & [-6.9; -6.9] & 488 & [-6.9; -6.6] & [36; 40] \\
				
				\rowcolor{gray!10!white}  40 & [-8.0; -8.0] & 1157 & [-8.1; -8.1] & 1157 & [-8.1; -8.1] & 572 & [-8.0; -7.7] & [42; 46] \\                                                                                                                                                   
			\end{tabular}                                                                                                                                             
			\caption{Guaranteed stability with probability greater than $0.99$ for \textbf{OWLS} and almost surely for the other methods.}
			\label{table:ex2Lprecision}
		\end{subtable}
		\begin{subtable}{.99\textwidth}
			\centering
			\begin{tabular}{cc|c|c|c||c|c|c}
				
				\rowcolor{gray!10!white}  & \multicolumn{4}{c||}{\textbf{Interpolation}} & \multicolumn{3}{c}{\textbf{Least-Squares regression}}\\
				& \multicolumn{1}{c|}{$\bm{\mathcal{I}}$\textbf{-GaussL}} & \multicolumn{1}{c|}{$\bm{\mathcal{I}}$\textbf{-Magic}} & \multicolumn{1}{c|}{$\bm{\mathcal{I}}$\textbf{-Leja}} & \multicolumn{1}{c||}{$\bm{\mathcal{I}}$\textbf{-Fekete}} & \multicolumn{1}{c|}{\textbf{OWLS}} & \multicolumn{1}{c|}{\makecell{{\textbf{BLS}} \\ $(M=100)$}} & \multicolumn{1}{c}{\textbf{s-BLS}}\\
				\rowcolor{gray!10!white} $p$ & $\varepsilon$ & $\varepsilon$ & $\varepsilon$ & $\varepsilon$ & $\varepsilon$ & $\varepsilon$ & $\varepsilon$\\
				5 & [-1.3; -1.3] &  [-0.9; -0.8] & [-1.1; -1.1] & [-1.1; -1.1] & [-0.6; 0.5] & [-1.1; -0.3] & [-1.2; -1.0]  \\                                                                                                                                                                                                         
				\rowcolor{gray!10!white}  10 & [-2.3; -2.3] & [-2.2; -2.2] & [-2.2; -2.2] & [-2.2; -2.2] & [-0.9; 1.1] & [-1.7; 0.6] & [-2.3; -1.6]  \\                                                                                                                                  
				
				15 & [-3.2; -3.1] & [-2.9; -2.8] & [-3.0; -2.9] & [-3.0; -2.9] & [-2.0; 1.5] & [-2.5; 1.2] & [-2.9; -2.6]  \\                                                                                                                                         
				
				\rowcolor{gray!10!white}  20 & [-4.2; -4.2] & [-4.0; -3.9] & [-4.1; -4.1] & [-4.1; -4.1] & [-1.9; 1.1] & [-2.8; 1.4] & [-4.1; -3.3]  \\                                                                                                                                                                                               
				25 & [-5.1; -5.0] &  [-4.9; -4.8] & [-4.9; -4.8] & [-4.8; -4.8] & [-2.4; 1.5] & [-3.5; 1.5] & [-4.8; -4.2]  \\                                                                                                                                                                                                
				\rowcolor{gray!10!white}  30 & [-6.1; -6.0] & [-5.7; -5.7] & [-6.0; -5.9] & [-6.0; -5.9] & [-3.2; 2.1] & [-4.0; 0.5] & [-5.8; -5.0] \\                                                                                                                                                                                                         
				35 & [-6.9; -6.9] &  [-6.6; -6.6] & [-6.7; -6.7] & [-6.7; -6.7] & [-4.3; 3.0] & [-4.0; 1.2] & [-6.7; -5.3] \\                                                                                                                                                                                                    
				\rowcolor{gray!10!white}  40 & [-7.9; -7.9] & [-7.7; -7.7] & [-7.8; -7.8] & [-7.8; -7.8] & [-5.6; 3.8] & [-4.7; 0.1] & [-7.7; -6.6]  \\                                                                                                              
				
			\end{tabular} 
			\caption{Given cost: $n= m$.}
			\label{table:ex2Lcost}
		\end{subtable}
		\caption{Approximation error $\varepsilon$ in log-10 scale for the example 2. Abbreviations are defined in \ref{subs:notations}.}
		\label{table:ex2L}
	\end{table}
	For this example, the same observations than in Subsection \ref{subsubs:first_function} can be made:
	\begin{itemize}
		\item when resampling (see Table \ref{table:ex2Lprecision}), it is possible to guarantee the stability of the approximation at a lower cost, without increasing the approximation error,
		\item when resampling and also subsampling (see Table \ref{table:ex2Lprecision}), it is possible to guarantee the stability of the approximation at a cost close to the interpolation regime, without increasing the approximation error,
		\item the \textbf{s-BLS} method is comparable to interpolation in terms of the accuracy of the approximation (see Table \ref{table:ex2Lcost}).
	\end{itemize}
	The only difference is that the $\bm{\mathcal{I}}$\textbf{-Magic} method behaves almost as well as the other $\bm{\mathcal{I}}$- method, which was not the case with the Gaussian measure.
	%The conclusions are in line with the observations we draw in sections \eqref{subsubs:first_function} and \eqref{subsubs:second_function}. However, we can outline some differences: for the first function, all methods (interpolation and least-squares) perform as well. In terms of number of samples, the greedy subsampling do not reach the interpolation regime. We expect that it would be possible with higher values of $M$. For the second function, this time the interpolation methods manage to recover the function in the biggest approximation space. 
	\subsubsection{A third function}
	We here consider the function
	\begin{equation}
	u_3(x) = \sum_{i=0}^p \exp\left(-\frac{i}{2}\right)\psi_i(x)
	\end{equation}
	where $\mathcal{X} = \mathbb{R}$ is equipped with the Gaussian measure, $(\psi_1, \hdots, \psi_m ) = \bm{U}(\varphi_1, \hdots, \varphi_m)$, with $\{\varphi_i\}_{i=1}^m$ the set of Hermite polynomials of degree less than $p$ and $\bm{U}$ an orthogonal matrix. In practice $\bm{U}$ is taken as the matrix of the left singular vectors of a $m \times m$ matrix $\bm{A}$, whose elements are i.i.d. realizations of a standard Gaussian random variable $\mathcal{N}(0,1)$.\\
	In this example, $p$ is chosen equal to 40, the approximation space $V_m = \text{span} \{ \psi_i :  1 \le i \le m \}$, and we consider different $\bm{U}$ for each trial. Therefore, we also have confidence intervals for the $\bm{\mathcal{I}}$- methods. The basis $(\psi_1, \hdots, \psi_m )$ is chosen as approximation basis. This is referred to as example 3 and the associated results are summarized in Table \ref{table:ex3precision} and Table \ref{table:ex3cost}. 
	\begin{table}
		\small
		\begin{subtable}{.99\textwidth}
			\centering
			\begin{tabular}{cc|c||c|c||c|c||c|c}
				& \multicolumn{2}{c||}{\textbf{OWLS}} & \multicolumn{2}{c||}{\makecell{{\textbf{c-BLS}} \\ $(M=1)$}} &  \multicolumn{2}{c||}{\makecell{{\textbf{c-BLS}} \\ $(M=100)$}} &  \multicolumn{2}{c}{\makecell{{\textbf{s-BLS}} \\ $(M=100)$}}\\
				\rowcolor{gray!10!white} $p$  & $\varepsilon$ & $n$ & $\varepsilon$ & $n$ & $\varepsilon$ & $n$ & $\varepsilon$ & $n$\\                                                                                                                                                                                                                                         
				5 & [-1.7; -1.5] & 134 & [-1.7; -1.5] & 134 & [-1.7; -1.4] & 48 & [-1.6; -1.1] & [6; 7] \\                 
				
				\rowcolor{gray!10!white}  10 & [-2.7; -2.6] & 265 & [-2.8; -2.6] & 265 & [-2.8; -2.6] & 108 & [-2.7; -2.3] & [11; 13] \\          
				15 & [-3.9; -3.7] & 405 & [-3.8; -3.7] & 405 & [-3.9; -3.7] & 176 & [-3.7; -3.3] & [18; 23] \\          
				
				\rowcolor{gray!10!white}  20 & [-5.0; -4.8] & 548 & [-5.0; -4.7] & 548 & [-5.0; -4.6] & 249 & [-4.9; -4.4] & [24; 26] \\            
				25 & [-6.1; -5.7] & 697 & [-6.0; -5.8] & 697 & [-6.0; -5.8] & 326 & [-6.0; -5.6] & [29; 34] \\              
				
				\rowcolor{gray!10!white}  30 & [-7.1; -6.9] & 848 & [-7.1; -6.9] & 848 & [-7.1; -6.9] & 405 & [-7.1; -6.7] & [34; 38] \\         
				35 & [-8.2; -7.9] & 1001 & [-8.2; -8.0] & 1001 & [-8.2; -7.9] & 488 & [-8.3; -7.7] & [40; 45] \\        
				
				\rowcolor{gray!10!white}  40 & [-15.8; -15.4] & 1157 & [-15.8; -15.5] & 1157 & [-15.8; -15.2] & 572 & [-15.7; -15.4] & [41; 47] \\
				
			\end{tabular}                                                                                                                                    
			\caption{Guaranteed stability with probability greater than $0.99$ for \textbf{OWLS} and almost surely for the other methods.}
			\label{table:ex3precision}
		\end{subtable}
		\begin{subtable}{.99\textwidth}
			\centering
			\begin{tabular}{cc|c|c|c||c|c|c}
				\rowcolor{gray!10!white}  & \multicolumn{4}{c||}{\textbf{Interpolation}} & \multicolumn{3}{c}{\textbf{Least-Squares regression}}\\
				& \multicolumn{1}{c|}{$\bm{\mathcal{I}}$\textbf{-GaussH}} & \multicolumn{1}{c|}{$\bm{\mathcal{I}}$\textbf{-Magic}} & \multicolumn{1}{c|}{$\bm{\mathcal{I}}$\textbf{-Leja}} & \multicolumn{1}{c||}{$\bm{\mathcal{I}}$\textbf{-Fekete}} & \multicolumn{1}{c|}{\textbf{OWLS}} & \multicolumn{1}{c|}{\makecell{{\textbf{BLS}} \\ $(M=100)$}} & \multicolumn{1}{c}{\textbf{s-BLS}}\\
				\rowcolor{gray!10!white} $p$ & $\varepsilon$ & $\varepsilon$ & $\varepsilon$ & $\varepsilon$ & $\varepsilon$ & $\varepsilon$ & $\varepsilon$\\
				5 & [-1.5; 0.2] & [-1.4; -0.7] & [-1.4; -0.8] & [-1.4; -1.0] & [-1.5; 0] & [-1.6; -1.0] & [-1.6; -1.1] \\                                                                                                                                           
				
				\rowcolor{gray!10!white} 10 & [-2.8; -1.2] & [-2.3; -1.6] & [-2.4; -1.8] & [-2.3; -1.5] & [-2.4; -1.3] & [-2.4; -1.4] & [-2.6; -1.7] \\                                                                                                                                 
				
				15 & [-3.6; -2.5]  & [-3.5; -2.6] & [-3.4; -2.4] & [-3.4; -2.7] & [-3.4; -0.9] & [-3.3; -1.6] & [-3.5; -1.0] \\                                                                                                                                 
				
				\rowcolor{gray!10!white}  20 & [-4.6; -3.6] & [-4.4; -3.6] & [-4.5; -3.9] & [-4.5; -4.0] & [-4.7; -3.8] & [-4.6; -3.4] & [-4.6; -3.8] \\                                                                                                                                
				
				25 & [-5.5; -4.2] & [-5.3; -4.7] & [-5.5; -4.8] & [-5.5; -4.9] & [-5.4; -3.3] & [-5.5; -4.1] & [-5.6; -5.1] \\                                                                                                                             
				
				\rowcolor{gray!10!white}  30 & [-6.8; -5.6] & [-6.5; -5.8] & [-6.4; -5.9] & [-6.5; -5.8] & [-7.2; -5.5] & [-6.6; -4.8] & [-6.8; -5.8] \\                                                                                                                                  
				
				35 & [-7.8; -6.5] & [-7.4; -6.5] & [-7.3; -6.9] & [-7.3; -6.9] & [-8.6; -6.8] & [-7.7; -6.4] & [-8.1; -7.2] \\                                                                                                                              
				
				\rowcolor{gray!10!white}  40 & [-15.9; -14.6] & [-8.2; -6.3] & [-13.2; -11.4] & [-12.4; -11.1] & [-11.8; -4.2] & [-14.2; -7.7] & [-15.7; -15.1] \\                                                                                                                                  
			\end{tabular} 
			\caption{Given cost: $n= m$.}
			\label{table:ex3cost}
		\end{subtable}
		\caption{Approximation error $\varepsilon$ for the example 3. Abbreviations are defined in \ref{subs:notations}.}
		\label{table:ex3}
	\end{table}
	Hence, in the same manner than in Table \ref{table:ex1H} and Table \ref{table:ex2L}, we notice that
	\begin{itemize}
		\item the error of approximation decreases when the size of the approximation space increases for all methods, we however notice that the $\bm{\mathcal{I}}$\textbf{-Magic} does not perform as well as the other methods, 
		\item the errors associated with \textbf{s-BLS} $(M=100)$ method are almost the same than the ones associated with the \textbf{OWLS}, \textbf{c-BLS} $(M=1)$ and \textbf{c-BLS} $(M=100)$ methods while being based on a number of evaluations of the function tending to the interpolation regime,
		\item the \textbf{s-BLS} method provides better results than the \textbf{OWLS} and \textbf{BLS} $(M=100)$ methods when $n$ is chosen equal to $m$ (interpolation regime). It also provides really better results than all the $\bm{\mathcal{I}}$-methods, except the $\bm{\mathcal{I}}$\textbf{-GaussH} (Gauss-Hermite points are still used).
	\end{itemize}
	
	For this example, it is important to notice that the approximation space is not generated from a set of commonly-used polynomials, for which there exists adapted sequences of points for interpolation. 
	This highlights the interest of the \textbf{s-BLS} method, which guarantees good sequences of points for the approximation, no matter what the approximation space is. The results obtained with the other  $\bm{\mathcal{I}}$-methods show that it is difficult to choose a suitable set of initial points (in size and distribution).\\
	\begin{remark}
		As in example \ref{subsubs:first_function}, the points for the $\bm{\mathcal{I}}$\textbf{-Magic}, $\bm{\mathcal{I}}$\textbf{-Fekete} and $\bm{\mathcal{I}}$\textbf{-Leja} are chosen among a sufficiently large and dense discretization of $\mathcal{X}$. Here, in the case where $\mathcal{X} = \mathbb{R}$, we choose a uniform discretization of the bounded interval $[-10, 10]$ with $10000$ points.
	\end{remark}
	\subsubsection{Multi-dimensional example}
	\label{subsubs:multi-d_example}
	In this section, we consider $\mathcal{X} = [-1,1]^d$, equipped with the uniform measure and the function:
	\begin{equation}
	u(x) = \frac{1}{1-\frac{0.5}{2d}\sum_{i=1}^dx_i}
	\end{equation}
	We consider the hyperbolic cross defined by
	\begin{equation}
	\mathbb{P}_{\Lambda} = \text{span}\{ \varphi_i(x) = \prod_{k=1}^{d} \psi^k_{i_k}(x), \ i \in \Lambda \} , \text{ where } \Lambda = \{ i = (i_1, \hdots i_d ), \  \prod_{k=1}^d (i_k+1) \le p+1\}
	\end{equation}
	where the $(\psi^k_{i_k})_{i_k \ge 0}$ are sequences of univariate Legendre polynomials.\\
	For each $k \in \{1, \hdots, d\}$, $ \{\psi^k_{i_k}\}_{i_k=1}^{k_{\max}}$ is an orthonormal basis of $L^2_{\mu_k}(\mathcal{X}_k)$.\\ To define a set of interpolation points unisolvent for the tensorized basis $\{\varphi_i\}_{i \in \Lambda}$, we proceed as follows, for each $1 \le k \le d$, we introduce a sequence of points $( z_{i_k}^k)_{i_k \ge 1} \in \mathcal{X}_k$ (Gauss, Leja, Fekete or Magic points) such that $( z_{i_k}^k)_{i_k \ge 1}^{p} $ is unisolvent for $\text{span} \{ \psi^k_{i_k} \}_{i_k = 1}^p$. Then we let, the multivariate sequence of points 
	$$ \Gamma_{\Lambda} = \{\bm{z}_{i} = (z^1_{i_1}, \hdots, z^d_{i_d}) : i \in \Lambda \} $$
	
	\begin{table}
		\small
		\begin{subtable}{.99\textwidth}
			\centering                                                                                                                                                                                                                                                                                                                             
			\begin{tabular}{ccc|c||c|c||c|c||c|c}                                                                                                                                                                                                                                                                                                      
				& 	&  \multicolumn{2}{c||}{\textbf{OWLS}} &  \multicolumn{2}{c||}{\makecell{{\textbf{c-BLS}} \\ $(M=1)$}} &  \multicolumn{2}{c||}{\makecell{{\textbf{c-BLS}} \\ $(M=100)$}} &  \multicolumn{2}{c}{\makecell{{\textbf{s-BLS}} \\ $(M=100)$}} \\                                                                                                                                                                    
				\rowcolor{gray!10!white}$p$ & $ m$ & $\varepsilon$ & $n$ & $\varepsilon$ & $n$ & $\varepsilon$ & $n$ & $\varepsilon$ & $n$ \\                                  
				4 & 10  & [-1.8; -1.7] & 238 & [-1.8; -1.8] & 238 & [-1.8; -1.8] & 96 & [-1.7; -1.5] & [10; 12] \\                                                                                                                                                           
				\rowcolor{gray!10!white}9 &  27 & [-3.3; -3.2] & 727 & [-3.3; -3.3] & 727 & [-3.3; -3.3] & 341 & [-3.2; -3.0] & [33; 38] \\             
				14 &  45 & [-4.2; -4.1] & 1282 & [-4.2; -4.2] & 1282 & [-4.2; -4.2] & 641 & [-4.1; -3.9] & [58; 63] \\                                                                                                                                                               
				\rowcolor{gray!10!white} 19 & 66 & [-5.6; -5.5] & 1960 & [-5.7; -5.5] & 1960 & [-5.7; -5.6] & 1019 & [-5.6; -5.5] & [92; 99] \\ 
				24 &  87 & [-6.5; -6.4] & 2659 & [-6.6; -6.4] & 2659 & [-6.5; -6.4] & 1418 & [-6.4; -6.3] & [130; 137] \\                                                                                                                                                        
				\rowcolor{gray!10!white} 29 & 111 & [-7.3; -7.1] & 3477 & [-7.3; -7.2] & 3477 & [-7.3; -7.2] & 1893 & [-7.2; -7.1] & [173; 183] \\                                                                                                                                                                             
			\end{tabular}                                                                                                                                                               
			\caption{Guaranteed stability with probability greater than $0.99$ for \textbf{OWLS} and almost surely for the other methods.}
			\label{table:ex4precision}
		\end{subtable}
		\begin{subtable}{.99\textwidth}
			\centering                        
			\begin{tabular}{ccc|c|c|c||c|c|c}
				\rowcolor{gray!10!white}&   & \multicolumn{4}{c||}{\textbf{Interpolation}} & \multicolumn{3}{c}{\textbf{Least-Squares regression}}\\
				& 	& \multicolumn{1}{c|}{$\bm{\mathcal{I}}$\textbf{-GaussL}} & \multicolumn{1}{c|}{$\bm{\mathcal{I}}$\textbf{-Magic}} & \multicolumn{1}{c|}{$\bm{\mathcal{I}}$\textbf{-Leja}} & \multicolumn{1}{c||}{$\bm{\mathcal{I}}$\textbf{-Fekete}} & \multicolumn{1}{c|}{\textbf{OWLS}} & \multicolumn{1}{c|}{\makecell{{\textbf{BLS}} \\ $(M=100)$}} & \multicolumn{1}{c}{\textbf{s-BLS}}\\
				\rowcolor{gray!10!white}       $p$ & $m$ & $\varepsilon$ & $\varepsilon$ & $\varepsilon$ & $\varepsilon$ & $\varepsilon$ & $\varepsilon$ & $\varepsilon$ \\

				4 & 10 &  [-1.0; -1.0] & [-1.0; -1.0]  & [-1.6; -1.6] & [-1.0; -1.0] & [-1.2; 0.5] & [-1.5; 0.5] & [-1.6; -1.4]  \\                                                                                                                                               
				
				\rowcolor{gray!10!white}9 & 27 &  [-1.7; -1.6] & [-2.2; -2.2] & [-3.1; -3.0] & [-1.7; -1.6] & [-2.7; -0.5] & [-2.7; -0.9] & [-2.9; -1.8]  \\                                                                                                                                  
				
				14 & 45 &  [-2.2; -2.1] & [-3.1; -3.0] & [-3.7; -3.6] & [-2.2; -2.1] & [-3.4; -1.5] & [-3.2; -2.0] & [-3.6; -2.4]  \\                                                                                                                                 
				
				\rowcolor{gray!10!white}19 &  66 &  [-2.7; -2.6] & [-4.0; -4.0] & [-5.4; -5.0] & [-2.7; -2.6] & [-4.8; 0.5] & [-4.5; -1.6] & [-5.1; -3.7] \\                                                                                                                                   
				
				24 & 87  & [-3.1; -3.0] &   [-4.1; -4.1] & [-6.1; -6.0] & [-3.0; -2.9] & [-5.3; -2.8] & [-5.2; -2.7] & [-5.6; -4.2]  \\                                                                                                                                       
				
				\rowcolor{gray!10!white}29 &  111  & [-1.2; -1.1] & [-4.9; -4.8] & [-6.7; -6.4] & [-1.0; -0.9] & [-5.8; -3.3] & [-5.9; -4.2] & [-6.5; -5.7]  \\                                                                                                                                                                                                                                                                                                                                     
			\end{tabular}  
			\caption{Given cost: $n= m$}
			\label{table:ex4cost}
		\end{subtable}
		\caption{Approximation error $\varepsilon$ for the example 4 with $d=2$. Abbreviations are defined in Subsection \ref{subs:notations}.}
		\label{table:ex4}
	\end{table}
	In Table \ref{table:ex4}, we observe that the best approximation errors are obtained with the least-squares regression methods, when the stability is guaranteed:
	\begin{itemize}
		\item in Table \ref{table:ex4precision}, the \textbf{s-BLS} method strongly reduces the number of samples necessary to get this stability, it is about $1.5$ times the interpolation regime.
		\item in Table \ref{table:ex4cost}, the $\bm{\mathcal{I}}$\textbf{-Leja} is the only interpolation method which performs better than the \textbf{s-BLS} method with a given cost. However, the $\bm{\mathcal{I}}$\textbf{-Leja} is less accurate than the \textbf{s-BLS} method with guaranteed stability.
	\end{itemize}
	\begin{table}
		\centering
		\small
		\begin{subtable}{.99\textwidth}
			\centering                                                                                                                                                                  
			\begin{tabular}{ccc|c||c|c||c|c||c|c}                                                                                                                                 
				& 	&  \multicolumn{2}{c||}{\textbf{OWLS}} &  \multicolumn{2}{c||}{\makecell{{\textbf{c-BLS}} \\ $(M=1)$}} &  \multicolumn{2}{c||}{\makecell{{\textbf{c-BLS}} \\ $(M=100)$}} &  \multicolumn{2}{c}{\makecell{{\textbf{s-BLS}} \\ $(M=100)$}} \\                                                                                                                                                                    
				\rowcolor{gray!10!white}$p$ & $ m$ & $\varepsilon$ & $n$ & $\varepsilon$ & $n$ & $\varepsilon$ & $n$ & $\varepsilon$ & $n$ \\                                                                                                                                                   
				4 & 23 & [-1.5; -1.4] & 608 & [-1.5; -1.5] & 608 & [-1.5; -1.5] & 279 & [-1.5; -1.3] & [27; 33] \\                                                                                                                                                            
				\rowcolor{gray!10!white} 7 & 63 & [-2.2; -2.0] & 1862 & [-2.2; -2.1] & 1862 & [-2.2; -2.0] & 963 & [-2.1; -1.9] & [99; 109] \\ 
				10 & 93 & [-2.2; -2.1] & 2861 & [-2.3; -2.1] & 2861 & [-2.3; -2.1] & 1535 & [-2.1; -2.0] & [164; 172] \\                                                                                                                                                        
				\rowcolor{gray!10!white} 13 & 153 & [-2.4; -2.3] & 4946 & [-2.5; -2.4] & 4946 & [-2.5; -2.4] & 2763 & [-2.4; -2.3] & [291; 305] \\ 
				
			\end{tabular} 
			\caption{Guaranteed stability with probability greater than $0.99$ for \textbf{OWLS} and almost surely for the other methods.}
			\label{table:ex5stability}
		\end{subtable}
		\begin{subtable}{.99\textwidth}
			\centering                        
			\begin{tabular}{ccc|c|c|c||c|c|c}
				
				\rowcolor{gray!10!white} &   & \multicolumn{4}{c||}{\textbf{Interpolation}} & \multicolumn{3}{c}{\textbf{Least-Squares regression}}\\
				&	& \multicolumn{1}{c|}{$\bm{\mathcal{I}}$\textbf{-GaussL}} & \multicolumn{1}{c|}{$\bm{\mathcal{I}}$\textbf{-Magic}} & \multicolumn{1}{c|}{$\bm{\mathcal{I}}$\textbf{-Leja}} & \multicolumn{1}{c||}{$\bm{\mathcal{I}}$\textbf{-Fekete}} & \multicolumn{1}{c|}{\textbf{OWLS}} & \multicolumn{1}{c|}{\makecell{{\textbf{BLS}} \\ $(M=100)$}} & \multicolumn{1}{c}{\textbf{s-BLS}}\\
				\rowcolor{gray!10!white}        $p$ & $m$ & $\varepsilon$ & $\varepsilon$ & $\varepsilon$ & $\varepsilon$ & $\varepsilon$ & $\varepsilon$ & $\varepsilon$ \\                                                                                                                                                                                                                                                                                                                                                                                                                                                                                          
				4 & 23 & [-0.6; -0.6] & [-1.0; -0.9] &  [-1.3; -1.2] & [-0.6; -0.6] & [-1.1; 0.7] & [-1.2; 0.1] & [-1.0; 0.1]  \\                                                                                                                                          
				
				\rowcolor{gray!10!white} 7 & 63 & [-0.9; -0.9] & [-1.4; -1.3] & [-1.2; -1.1] & [-0.9; -0.9] & [-1.2; -0.2] & [-1.7; 0.1] & [-1.6; -0.5]  \\                                                                                                                                    
				10 & 93 & [-1.0; -0.9] & [-1.3; -1.3] & [-0.5; -0.4] & [-1.0; -0.9] & [-1.2; 0] & [-1.1; 1.0] & [-1.4; 0.4]  \\                                                                                                                                            
				\rowcolor{gray!10!white}13 & 153 & [-1.2; -1.1] & [-1.6; -1.5] & [-2.2; -2.0] & [-1.2; -1.1] & [-1.7; -0.2] & [-1.5; -0.3] & [-1.6; -0.7]  \\                                                                                                                                                                                                                                                                                                                                                                                                                                  
			\end{tabular}                                                                                                                                                                                   
			\caption{Given cost: $n= m$}
			\label{table:ex5cost}
		\end{subtable}
		\caption{Approximation error $\varepsilon$ for the example 4 with $d=4$. Abbreviations are defined in Subsection \ref{subs:notations}.}
		\label{table:ex5}
	\end{table}
	In Table \ref{table:ex5}, we observe that the conclusions are the same in dimension 4 than for the dimension 2. The only difference is that the number of samples necessary to get the stability is about 2 times the interpolation regime.
	\subsection{A noisy example}
	In this example, we consider the case where the evaluations are affected by a noise $e$, which fulfils the assumptions of Section \ref{sec:noisy_case} (the noise is independent from the sample $\widetilde{\bm{x}}^n_K$, centred with a bounded conditional variance). Here we choose $e \sim \mathcal{N}(0, \sigma^2)$.\\
	\begin{table}
		\small
		
		\begin{subtable}{.99\textwidth}
			\centering                                                                                                                                                                  
			\begin{tabular}{ccc|c||c|c||c|c||c|c}                                                                                                                                 
				& 	&  \multicolumn{2}{c||}{\textbf{OWLS}} &  \multicolumn{2}{c||}{\makecell{{\textbf{c-BLS}} \\ $(M=1)$}} &  \multicolumn{2}{c||}{\makecell{{\textbf{c-BLS}} \\ $(M=100)$}} &  \multicolumn{2}{c}{\makecell{{\textbf{s-BLS}} \\ $(M=100)$}} \\                                                                                                                                                                    
				\rowcolor{gray!10!white}					$m$ & $\sigma$ & $\varepsilon$ & $N$ & $\varepsilon$ & $N$ & $\varepsilon$ & $N$ & $\varepsilon$ & \# $\bm{K}$ \\                                                                                                                  
				10 & 0.1 & [-1.8; -1.8] & 238 & [-1.8; -1.7] &  238 & [-1.8; -1.8] & 238 & [-1.6; -1.3] & [10; 11] \\                                                                                                                            
				& 0.01 & [-1.8; -1.7] & 238 & [-1.8; -1.6] &  238 &  [-1.8; -1.8] & 238 & [-1.6; -1.3] & [10; 11] \\                                                                                                                             
				& 0.001 & [-1.8; -1.6] & 238 & [-1.8; -1.8] &  238 &  [-1.8; -1.7] & 238  & [-1.7; -1.4] & [10; 12] \\                                                                                                                          
				\rowcolor{gray!10!white} 					27 & 0.1 & [-2.8; -2.6] & 727 & [-2.8; -2.6] & 727 & [-2.8; -2.6] & 727 & [-2.0; -1.8] & [31; 35] \\                                                                                                                            
				\rowcolor{gray!10!white} 					& 0.01 & [-3.3; -3.2] & 727 & [-3.3; -3.3] & 727 & [-3.3; -3.3] & 727 &  [-3.2; -3.0] & [30; 37] \\                                                                                                                             
				\rowcolor{gray!10!white} 					& 0.001 & [-3.3; -3.3] & 727 & [-3.4; -3.3] & 727 & [-3.3; -3.3] & 727 & [-3.2; -2.9] & [31; 36] \\
			\end{tabular} 
			\caption{Guaranteed stability with probability greater than $0.99$ for \textbf{OWLS} and almost surely for the other methods.}
			\label{table:exnoisystability}
		\end{subtable}
		\begin{subtable}{.99\textwidth}
			\centering       
			\begin{tabular}{ccc|c|c|c||c|c|c}
				
				\rowcolor{gray!10!white} &   & \multicolumn{4}{c||}{\textbf{Interpolation}} & \multicolumn{3}{c}{\textbf{Least-Squares regression}}\\
				&	& \multicolumn{1}{c|}{$\bm{\mathcal{I}}$\textbf{-GaussL}} & \multicolumn{1}{c|}{$\bm{\mathcal{I}}$\textbf{-Magic}} & \multicolumn{1}{c|}{$\bm{\mathcal{I}}$\textbf{-Leja}} & \multicolumn{1}{c||}{$\bm{\mathcal{I}}$\textbf{-Fekete}} & \multicolumn{1}{c|}{\textbf{OWLS}} & \multicolumn{1}{c|}{\makecell{{\textbf{BLS}} \\ $(M=100)$}} & \multicolumn{1}{c}{\textbf{s-BLS}}\\
				\rowcolor{gray!10!white}					$m$  & $\sigma$ & $\varepsilon$ & $\varepsilon$ & $\varepsilon$ & $\varepsilon$ & $\varepsilon$ & $\varepsilon$ & $\varepsilon$ \\                                                                                                                                                                                                                                                         
				10 & 0.1 & [-1.3; -0.5] & [-1.1; -1.0] & [-1.3; -0.5] & [-1.2; -0.7] & [-0.8; 2] & [-1.4; 0.0] & [-1.6; -1.1] \\                                                                                                                        
				& 0.01 & [-1.0; -0.9] & [-1.0; -1.0] & [-1.0; -0.9] & [-1.0; -0.9] & [-1.3; 1.2] & [-1.5; -0.1] & [-1.6; -1.2] \\                                                                                                                                
				& 0.001 & [-1.0; -0.9] & [-1.0; -1.0] & [-1.0; -0.9] & [-1.0; -0.9] & [-1.5; 1.0] & [-1.6; -0.1] &[-1.6; -1.0] \\                                                                                                                                   
				
				\rowcolor{gray!10!white}				27 & 0.1 & [1.5; 2.2] & [-0.8; 0.1] & [1.5; 2.2] & [1.4; 2.2] & [0.8; 7.6] & [-0.7; 1.7] & [-1.8; -0.7] \\                                                                                                                              
				
				\rowcolor{gray!10!white}				& 0.01 & [-0.2; 0.3] & [-2.5; -1.5] & [-0.2; 0.3] & [-0.2; 0.3] & [-1.4; 6.8] & [-2.3; 0.4] & [-3.0; -2.3] \\                                                                                                                     
				\rowcolor{gray!10!white}				& 0.001 & [-2.0; -1.4] & [-1.9; -1.9] & [-2.0; -1.4] & [-2.0; -1.5] & [-2.2; -0.3] & [-2.8; -0.6] & [-2.9; -1.9] \\                                                                                                                                                                  
			\end{tabular} 
			\caption{Given cost: $n= m$}
			\label{table:exnoisycost}
		\end{subtable}
		\caption{Approximation error $\varepsilon$ for the example 4 noisy with $d=2$. Abbreviations are defined in Subsection \ref{subs:notations}.}
		\label{table:exnoisy}
	\end{table}
	
	Table \ref{table:exnoisy} presents the obtained results for two different approximation spaces' dimensions ($m=10, 27$) and $3$ different standard deviations of the noise ($\sigma = 0.1, \ 0.01, \ 0.001$). When the stability is guaranteed, see Table \ref{table:exnoisystability}, we observe that the influence of the noise is more important in the $\textbf{s-BLS}$ method (which is in line with the expression from \ref{quasi-optim-greedy-noisy}, the higher $n$ is the smaller the contribution due the noise is). Furthermore, when $m=10$ ($p=4$), the approximation error is determined by the dimension of the approximation space, whereas when $m=27$ ($p=9$), the approximation error decreases when $\sigma$ decreases.\\
	In Table \ref{table:exnoisycost}, we observe that the $\bm{\mathcal{I}}$-methods are not robust to the noise, the approximation error is much more higher than for the $\textbf{s-BLS}$ method performed with $m=n$, for both dimensions of approximation spaces.
	\subsection{Overall conclusion for all examples}
	When the stability is guaranteed (see Table \ref{table:ex1Hprecision}, Table \ref{table:ex2Lprecision}, Table \ref{table:ex3precision}, Table \ref{table:ex4precision}), the $\textbf{s-BLS}$ method, whose cost is close to $m$, behaves in the same way than the other least-squares regression methods and the best $\bm{\mathcal{I}}$-method.
	When $n=m$, the $\textbf{s-BLS}$ method performs better than most of the interpolation methods.
	Depending on the choice of the approximation basis or the dimension of the problem $d$, the $\bm{\mathcal{I}}$-methods are more or less efficient. None of them give the best results in all situations, whereas the $\textbf{s-BLS}$ method with stability guaranteed is as efficient as the best $\bm{\mathcal{I}}$-method while having a number of samples tending to $m$.\\
	The example where the data are polluted with noise shows one major interest of the $\textbf{s-BLS}$ method compared to the $\bm{\mathcal{I}}$-methods: it is more robust.\\
	\section{Conclusion}
	We have proposed a method to construct the projection of a function $u$ in a given approximation space $V_m$ with dimension $m$. In this method, the approximation is a weighted least-squares projection associated with random points sampled from a suitably chosen distribution. We obtained quasi-optimality properties (in expectation) for the weighted least-squares projection, with or without reducing the size of the sample by a greedy removal of points.\\
	The error bound in the quasi-optimality property depends on the number of points selected by the greedy algorithm. The more points removed, the larger the bound will be. Therefore, if the goal is an accurate control of the error, as few points as possible should be removed. On the contrary, if the goal is to reduce the cost as much as possible but allows a larger bound of the error, the maximum number of points may be removed from the sample, which in some cases leads to an interpolation regime $(n=m)$.\\
	Furthermore, the bound obtained for the approximation error in the noisy case shows that the error is partly determined by the noise level. Depending on the goal, the influence of this noise can be reduced if necessary by the number of points.\\
	As the convergence of this greedy algorithm to the interpolation regime is not systematic, it would be interesting to look for an optimal selection of the sub-sample with regard to the stability criterion.\\
	With this method, the points are sampled from a distribution which depends on the approximation space. Considering strategies where this approximation space is chosen adaptively, as in \cite{CohenBachmayr2018} or \cite{Migliorati2018}, an important issue is the reuse of samples from one approximation space to another.
	
	\appendix 
	
	\section{Proof of Lemma \ref{lem:minimumexpectation}}
	Recall that for any sample $\bm x^n$, $Z_{\bm x^n} = \Vert \bm{G}_{\bm{x}^{n}} - \bm I \Vert_2$ and $\mathbb{P}(A_\delta) \ge 1-\eta^M$ (Lemma \ref{lem:resampling-proba}). By definition of $\bm{x}^{n,{\star}}$, we have  $\bm{x}^{n,{\star}} = \bm{x}^{n,{I^\star}}$, where given the $M$ samples $\bm x^{n,1},\hdots,\bm x^{n,M}$, the random variable 
	$I^\star$ follows the uniform distribution on the set 
	$\arg \min_{1 \le i \le M} Z_{\bm{x}^{n,i}}$ (possibly reduced to a singleton). 
	The property \eqref{norm-tildexn-bound} is a particular case of \eqref{norm-tildexn-bound-epsilon} for $\varepsilon=1$. However, let us first provide a simple proof of \eqref{norm-tildexn-bound}. We have
	\begin{equation*}
	\begin{aligned}
	\mathbb{E}\left(\Vert v \Vert_{\widetilde {\bm{x}}^{n}}^2\right) &= \mathbb{E}\left(\Vert v \Vert_{\bm{x}^{n,\star}}^2 \vert A_\delta\right) \le \mathbb{E}\left(\Vert v \Vert_{\bm{x}^{n,\star}}^2  \right)\mathbb{P}(A_\delta)^{-1} 
	\\
	&\le \mathbb{E}\left(\Vert v \Vert_{\bm{x}^{n,I^\star}}^2  \right) (1-\eta^M)^{-1}
	\le \sum_{j=1}^M \mathbb{E}\left(\Vert v \Vert_{\bm{x}^{n,j}}^2  \right) (1-\eta^M)^{-1} \\
	&= \Vert v \Vert^2 M (1-\eta^M)^{-1}.
	\end{aligned}
	\end{equation*}
	Now let us consider the proof of the other inequalities. We first note that $A_{\delta} = \{Z_{\bm x^{n,I^\star}} \le \delta\}= \left\{\min_{1 \le i \le M} Z_{\bm{x}^{n,i}} \le \delta \right\}.
	$
	We consider the events $B_{j} = \{I^\star = j \}$ which form a complete set of events. From the definition of $I^\star$ and $A_\delta$ and the fact that the samples $ \bm{x}^{n,i}$ are i.i.d., it is clear that $\mathbb{P}({B_j}) = \mathbb{P}(B_1) = M^{-1}$ and 
	$\mathbb{P}(B_j \cap A_\delta)=\mathbb{P}(B_1 \cap A_\delta)$ for all $j$. Therefore, 
	$$
	\mathbb{P}(A_{\delta} \cap B_1) = \frac{1}{M}  \sum_{j=1}^M \mathbb{P}(A_{\delta} \cap B_j) = \frac{1}{M}  \mathbb{P}(A_{\delta}) \ge  \frac{(1-\eta^M)}{M}.
	$$
	Then 
	\begin{equation*}
	\begin{aligned}
	\mathbb{E}\left(\Vert v \Vert_{\widetilde {\bm{x}}^{n}}^2\right) &= \mathbb{E}\left(\Vert v \Vert_{\bm{x}^{n,\star}}^2 \vert A_\delta\right) = \sum_{j=1}^M \mathbb{E}\left( \Vert v \Vert^2_{\bm{x}^{n,j}} \vert A_\delta \cap B_j\right) \mathbb{P}(B_j) = \mathbb{E}\left( \Vert v \Vert^2_{\bm{x}^{n,1}} \vert A_\delta \cap B_1\right)\\
	&=  \mathbb{E}\left( \Vert v \Vert^2_{\bm{x}^{n,1}} \mathds{1}_{A_\delta \cap B_1}\right)\mathbb{P}(A_{\delta} \cap B_1)^{-1}\\
	& \le \mathbb{E}\left( \Vert v \Vert^2_{\bm{x}^{n,1}}  \mathds{1}_{ Z_{\bm{x}^{n,1}} \le \delta} \mathds{1}_{\min_{2 \le i \le M} Z_{\bm{x}^{n,i}} \ge Z_{\bm{x}^{n,1}}} \right){M} (1 -\eta^M)^{-1}\\
	& = \mathbb{E}\left( \Vert v \Vert^2_{\bm{x}^{n,1}}  \mathds{1}_{ Z_{\bm{x}^{n,1}} \le \delta} \mathbb{E}\left(\mathds{1}_{\min_{2 \le i \le M} Z_{\bm{x}^{n,i}} \ge Z_{\bm{x}^{n,1}}} \vert \bm{x}^{n,1}\right)\right){M} (1 -\eta^M)^{-1}\\
	& = \mathbb{E}\left( \Vert v \Vert^2_{\bm{x}^{n,1}}  \mathds{1}_{ Z_{\bm{x}^{n,1}} \le \delta} \mathbb{E}\left(\mathds{1}_{Z_{\bm x^{n,2}} > Z_{\bm{x}^{n,1}}} \vert \bm{x}^{n,1}\right)^{M-1}\right){M} (1 -\eta^M)^{-1}.
	\end{aligned}
	\end{equation*}
	Using H\"older's inequality, we have that for any $0 < \varepsilon \le 1$,  
	\begin{equation*}
	\small
	\begin{aligned}
	\mathbb{E}\left(\Vert v \Vert_{\bm{x}^{n,\star}}^2 \vert A_\delta\right) & \le \mathbb{E}\left(\Vert v \Vert^{\frac{2}{\varepsilon}}_{\bm{x}^{n,1}}  \mathds{1}_{ Z_{\bm{x}^{n,1}} \le \delta}\right)^{\varepsilon} \mathbb{E}\left( \mathbb{E} \left( \mathds{1}_{Z_{\bm x^{n,2}} > Z_{\bm{x}^{n,1}}} \vert \bm{x}^{n,1}\right)^{\frac{M-1}{1-\varepsilon}}\right)^{1 - \varepsilon} {M} (1 -\eta^M)^{-1}\\
	&  \le \mathbb{E}\left(\Vert v \Vert^{\frac{2}{\varepsilon}}_{\bm{x}^{n,1}}\right)^{\varepsilon} \mathbb{E}\left( \mathbb{E} \left( \mathds{1}_{Z_{\bm x^{n,2}} > Z_{\bm{x}^{n,1}}} \vert \bm{x}^{n,1}\right)^{\frac{M-1}{1-\varepsilon}}\right)^{1 - \varepsilon} {M} (1 -\eta^M)^{-1}.
	\end{aligned}
	\end{equation*}
	For any measurable function $f$ and any two i.i.d. random variables $X$ and $Y$, we have that $\mathbb{E}(\mathds{1}_{f(X) >f(Y)} \vert Y)$ is a uniform random variable on $(0,1)$. Therefore 
	$\mathbb{E} \left( \mathds{1}_{Z_{\bm x^{n,2}} > Z_{\bm{x}^{n,1}}} \vert \bm{x}^{n,1}\right)$ is uniformly distributed on $\left(0,1\right)$ and
	$$\mathbb{E}\left( \mathbb{E} \left( \mathds{1}_{Z_{\bm x^{n,2}} > Z_{\bm{x}^{n,1}}} \vert \bm{x}^{n,1}\right)^{\frac{M-1}{1-\varepsilon}}\right) = \frac{1}{\frac{M-1}{1 - \varepsilon} + 1} = \frac{1 - \varepsilon}{M- \varepsilon}.
	$$   
	By combining the previous results, we obtain 
	\begin{equation*}
	\begin{aligned}
	\mathbb{E}\left(\Vert v \Vert_{\bm{x}^{n,\star}}^2 \vert A_\delta\right) & \le \mathbb{E}\big(\Vert v \Vert^{\frac{2}{\varepsilon}}_{\bm{x}^{n}}\big)^{\varepsilon} {M} (1 -\eta^M)^{-1} \frac{(1- \varepsilon)^{1 - \varepsilon}}{(M -\varepsilon)^{1 - \varepsilon}}.
	%& \le \mathbb{E}\left(\Vert v \Vert^{\frac{2}{\varepsilon}}_{\bm{x}^{n,1}}\right)^{\varepsilon} \frac{M^{\varepsilon}}{(1 - \eta^M)}
	\end{aligned}
	\end{equation*}
	For $\varepsilon= 1$, we recover the result \eqref{norm-tildexn-bound}. 
	The last result simply follows from 
	$$
	\mathbb{E}\big(\Vert v \Vert^{\frac{2}{\varepsilon}}_{\bm{x}^{n}}\big) \le  \mathbb{E}\left(\Vert v \Vert^2_{\bm{x}^{n}}\right) \Vert v \Vert_{\infty,w}^{2/\varepsilon-2} =   \Vert v \Vert^2 \Vert v \Vert_{\infty,w}^{2/\varepsilon-2}.
	$$
	\section{Approximate fast greedy algorithm}
	\label{alg:greedy_strategy}
	\subsection{Computational strategy for the approximate fast greedy algorithm}
	For any $k$, we have
	\begin{equation}
	\bm{G}_{{\bm{x}^n_{K \setminus \{k\}}}} = \frac{\# K}{\# K-1}\bm{G}_{{\bm{x}^n_K}} - \frac{1}{\# K-1} w(x^{k})\bm{\varphi}(x^{k})\otimes \bm{\varphi}(x^{k})
	\end{equation}
	and 
	\begin{equation}
	\bm{G}_{{\bm{x}^n_{K \setminus \{k\}}}} - \bm{I} = \frac{\# K}{\# K-1}\left(\bm{G}_{{\bm{x}^n_K}} - \bm{I}\right)  - \frac{1}{\# K-1} w(x^{k})\bm{\varphi}(x^{k})\otimes \bm{\varphi}(x^{k}) + \frac{1}{\# K-1}\bm{I}.
	\end{equation}
	\red{We let $$\bm{A}:=  \frac{\# K}{\# K-1}\left(\bm{G}_{{\bm{x}^n_K}} - \bm{I}\right),  \quad \bm{v}^k: = \left(\frac{1}{\# K-1} w(x^{k})\right)^{1/2} \bm{\varphi}(x^{k}), ,$$
	so that $$\bm{G}_{{\bm{x}^n_{K \setminus \{k\}}}} - \bm{I}  = \bm{B} + a \bm{I}, \quad \text{with} \quad  \bm{B} := \bm{A} - \bm{v}^k \otimes \bm{v}^k \quad \text{and} \quad  a=\frac{1}{\# K-1}.$$ 
	It holds}
	$$ \Vert \bm{G}_{{\bm{x}^n_{K \setminus \{k\}}}} - \bm{I} \Vert_2 = \max(\lambda_1(\bm{B}) \red{+a}, -\lambda_m(\bm{B}) \red{-a}).$$
	As the matrix $\bm{B}$ is a rank-one update of the symmetric matrix $\red{\bm{A}}$, 
	%from \cite{Golub1973} and \cite{Bunch1978}, 
	we can state that
	\begin{equation}
	\lambda_{1}(\bm{B}) \le \red{\lambda_1(\bm{A})}
	%\frac{\# K}{\# K-1}\lambda_{1}(\bm{G}_{{\bm{x}^n_K}} - \bm{I}),
	\end{equation}
	where $\lambda_{1}(\bm{B})$ and $\lambda_{1}(\red{\bm{A}})$ are the maximal eigenvalues of $\bm{B}$ and \red{$\bm{A}$}, respectively.\\
	Bounds on highest and lowest eigenvalues are obtained in \cite{Ipsen2009} or \cite{Benasseni2011}. Here we consider the lower bounds from \cite{Benasseni2011}. For the highest eigenvalue
		$$\lambda_1(\bm{B}) \ge \red{\lambda_1(\bm{A}) - (\bm{q}_1^T \bm{v}^k)^2},$$ where $\bm{q}_1 $ is the \red{normalized} eigenvector of $\bm{B}$ associated to its highest eigenvalue $\lambda_1(\bm{B})$.\\
To bound the quantity $-\lambda_m(\bm{B}) \red{-a}$, we use the fact that $$-\lambda_{m}(\bm{B})  = \lambda_{1}(-\bm{B}) \red{=  \lambda_1(-\bm{A} + \bm{v}^k \otimes \bm{v}^k )}$$ and the bound  
	$$\red{ \lambda_1(-\bm{A} + \bm{v}^k \otimes \bm{v}^k ) \ge  -\lambda_m(\bm{A}) +  ( \bm{q}_m^T \bm{v}^k)^2}
	 %\frac{\# K}{\# K-1}\lambda_{1}(-(\bm{G}_{{\bm{x}^n_K}} - \bm{I})) + \widetilde{\rho}(\bm{q}_1^T\bm{\varphi}(x^{k})),
	 $$
	 \red{where $\bm{q}_m$ is the \red{normalized} eigenvector of $\bm{B}$ associated to its lowest eigenvalue $\lambda_m(\bm{B})$. We deduce  the lower bound 
	 $$
	 \Vert \bm{G}_{{\bm{x}^n_{K \setminus \{k\}}}} - \bm{I} \Vert_2 \ge \max\{\lambda_1(\bm{A}) - (\bm{q}_1^T \bm{v}^k)^2 +a ,-\lambda_m(\bm{A}) +  ( \bm{q}_m^T \bm{v}^k)^2 - a \}.
	 $$
	 }
%	such that
%	$$\lambda_{m}(-\bm{B}) \ge -\frac{\# K}{\# K-1}\lambda_{m}((\bm{G}_{{\bm{x}^n_K}} - \bm{I})) + \widetilde{\rho}(\bm{q}_1^T\bm{\varphi}(x^{k})).$$
Instead of calculating $\Vert \bm{G}_{{\bm{x}^n_{K \setminus \{k\}}}} - \bm{I} \Vert_2$ for each $k$, \red{we evaluate $\bm{q}_1^T\bm{v}^k$ and $\bm{q}_m^T\bm{v}^k$}, this operation only
involves matrix multiplications, and we look for $k_1$ which \red{minimizes $\lambda_1(\bm{A}) - (\bm{q}_1^T \bm{v}^k)^2 +a $ and $k_2$ which minimizes $-\lambda_m(\bm{A}) +  ( \bm{q}_m^T \bm{v}^k)^2- a$}. Then we choose $$k^{\star} \in \arg\min_{k \in \{k_1, k_2\}} \Vert \bm{G}_{{\bm{x}^n_{K \setminus \{k\}}}} - \bm{I} \Vert_2.$$
	We may not find the minimizer over $K$, but in practice, we observe that this approach performs almost as well as the exact greedy and it is faster.
	\subsection{Complexity analysis}
	We compare these two subsampling methods in terms of computational cost. For the exact greedy subsampling, each time we remove a point from a $l$-sample, it requires $l$ calculations of $\Vert \bm{G}_{\bm{x}^l} -\bm{I} \Vert_2$, which takes $\mathcal{O}(m^3)$ floating point operations. Furthermore, the computation of $\bm{G}_{\bm{x}^l}$ requires $\mathcal{O}(m^2l^2)$ floating point operations. 
	We start from $l=n$ and let the greedy subsampling runs until the stability condition is no longer verified to a sample of size $k$, summing over all withdrawn points it comes
	$$\sum_{l=k+1}^{n} m^3l+m^2l^2 = \frac{m^3}{2}(n(n+1)-k(k+1)) + \frac{m^2}{6}(n(n+1)(2n+1)-k(k+1)(2k+1)).$$
	Assuming $k = c m $ with $c \ge 1$, the overall cost scales in
	$$ \mathcal{C}_E = \mathcal{O}(m^2n^3).$$
	The more points are withdrawn ($c$ smaller), the sharper this bound is.\\
	For the fast greedy subsampling method, each time we remove a point from a $l$-sample it requires one singular value decomposition of $\bm{G}_{\bm{x}^l} -\bm{I}$, which takes $\mathcal{O}(m^3)$ floating point operations. It also requires $\mathcal{O}(m^2l)$ for the computation of $\bm{G}_{\bm{x}^l}$. Using the same assumptions than before, $k = c m $ with $c \ge 1$, the overall cost scales in
	$$ \mathcal{C}_F = \mathcal{O}(m^2n^2).$$
	The more points are withdrawn ($c$ smaller), the sharper this bound is.\\
	In regards to the assumptions made in the Theorem \ref{th:s-BLS-accuracy}, we can say $n = \mathcal{O}(m\log(m))$ and therefore, 
	$$ \mathcal{C}_E = \mathcal{O}(m^5\log(m)^3) \text{ and }  \mathcal{C}_F = \mathcal{O}(m^4\log(m)^2).$$
	
	\subsection{Illustration}
	
	In the next table, we present the CPU computational times for the subsampling part, when using the technique presented in Algorithm \ref{alg:greedy_strategy} compared to an exact greedy approach, we also illustrate that its accuracy is the same by considering the example 2, with $\mathcal{X} = [-1, 1]$ equipped with the uniform measure and the function
	\begin{equation}
	u(x) = \frac{1}{1+5x^2}.
	\end{equation}		
	The approximation space is $V_m = \mathbb{P}_{m-1} = \text{span}\{\varphi_i: 1 \le i \le m\}$, where the basis $\{\varphi_i\}_{i=1}^m$ is chosen as the Legendre polynomials of degree less than $m-1$.\\ 
	\begin{table}
		\small
		\centering                                                                                                                                                               
		\begin{tabular}{c||c|c|c|c||c|c|c|c}  
			\rowcolor{gray!10!white}	& \multicolumn{4}{c||}{\textbf{Exact Greedy Subsampling}} & \multicolumn{4}{c}{\textbf{Fast Greedy Subsampling}}\\                                                                                                                                                                                                                                                                                         
			$m$ & $\Vert \bm{G}_{{\bm{x}^n_{K}}} - \bm{I} \Vert_2$ & $\#K$ & $\varepsilon(u^{\star})$ &  t (sec) & $\Vert \bm{G}_{{\bm{x}^n_{K}}} - \bm{I} \Vert_2$ & $\#K$ & $\varepsilon(u^{\star})$ & t (sec)\\                                                                                        
			\rowcolor{gray!10!white} 6 & [0.34; 0.68] & [6; 6] & [-0.9; -0.7] & [0.29; 0.43] & [0.18; 0.71] & [6; 6] & [-1; -0.6] & [0.03; 0.09] \\                                                                                                                                                                             
			11 & [0.42; 0.76] & [11; 11] & [-2; -1.7] & [1.43; 1.48] & [0.41; 0.67] & [11; 12] & [-2.1; -1.9] & [0.16; 0.22] \\                                                                                                                                                                     
			\rowcolor{gray!10!white} 16 & [0.47; 0.87] & [16; 16] & [-2.8; -2.4] & [13; 15] & [0.60; 0.88] & [16; 17] & [-2.8; -2.3] & [0.52; 0.61] \\                                                                                                                                                                   
			21 & [0.63; 0.87] & [21; 21] & [-4; -3.7] & [52; 54] & [0.74; 0.88] & [21; 23] & [-4; -3.6] & [1.6; 1.9] \\                                                                                                                                                                       
			\rowcolor{gray!10!white} 26 & [0.67; 0.88] & [26; 26] & [-4.6; -4.4] & [150; 204] & [0.64; 0.89] & [26; 29] & [-4.8; -4.4] & [3.5; 5.3] \\                                                                                                                                                              
			31 & [0.79; 0.89] & [31; 31] & [-5.7; -5.5] & [429; 464] & [0.73; 0.89] & [31; 34] & [-5.9; -5.5] & [5.8; 8.6] \\                                                                                                                                                                  
			\rowcolor{gray!10!white} 36 & [0.68; 0.87] & [36; 36] & [-6.5; -6.3] & [608; 1008] & [0.70; 0.87] & [36; 40] & [-6.7; -6.4] & [8.6; 12] \\                                                                                                                                                             
			41 & [0.77; 0.88] & [41; 42] & [-7.6; -7.4] & [948; 999] & [0.74; 0.89] & [42; 48] & [-7.8; -7.4] & [11; 13] \\
		\end{tabular}  
		\caption{Comparison of the performance between exact and fast greedy approaches, using $\delta = 0.9$ and $\eta =0.01$ both for different $m$. Number of samples $\#K$, stability constant $\Vert \bm{G}_{{\bm{x}^n_{K}}} - \bm{I} \Vert_2$ and CPU time $t$.}  
		\label{table:greedy_time_uniform}                                                                                                                                                        
	\end{table}   
	\begin{table}
		\small                                                                                                                                                           
		\centering                                                                                                                                                               
		\begin{tabular}{c||c|c|c||c|c|c}  
			\rowcolor{gray!10!white}	& \multicolumn{3}{c||}{\textbf{Exact Greedy Subsampling}} & \multicolumn{3}{c}{\textbf{Fast Greedy Subsampling}}\\                                                                                                                                                                                                                                                                                            
			$m$ & $\Vert \bm{G}_{{\bm{x}^n_{K}}} - \bm{I} \Vert$ & $\#K$  &  t (sec)  & $\Vert \bm{G}_{{\bm{x}^n_{K}}} - \bm{I} \Vert$ & $\#K$  & t (sec) \\                                                                                                                                                                                                                   
			\rowcolor{gray!10!white} 6 & [0.32; 0.63] & [6; 6] & [0.28; 0.37] & [0.21; 0.85] & [6; 6] & [0.03; 0.07] \\                                                                                                                                                                           
			11 & [0.48; 0.82] & [11; 11] & [1.45; 1.48] & [0.50; 0.77] & [11; 11] & [0.175; 0.22] \\                                                                                                                                                                      
			\rowcolor{gray!10!white} 16 & [0.61; 0.85] & [16; 16] & [13.9; 14.8] & [0.56; 0.87] & [16; 17] & [0.53; 0.6] \\                                                                                                                                                                     
			21 & [0.65; 0.84] & [21; 21] & [60; 68] & [0.62; 0.86] & [21; 23] & [1.79; 2.17] \\                                                                                                                                                                    
			\rowcolor{gray!10!white} 26 & [0.72; 0.87] & [26; 26] & [173; 186] & [0.68; 0.88] & [26; 28] & [3.5; 4.2] \\                                                                                                                                                             
			31 & [0.79; 0.87] & [31; 32] & [354; 377] & [0.66; 0.87] & [32; 35] & [5.2; 6.5] \\                                                                                                                                                                      
			\rowcolor{gray!10!white} 36 & [0.74; 0.87] & [36; 37] & [584; 636] & [0.64; 0.88] & [36; 42] & [8.2; 8.8] \\                                                                                                                                                                    
			41 & [0.73; 0.87] & [41; 42] & [993; 1041] & [0.74; 0.88] & [41; 44] & [11.9; 12.2] \\  
		\end{tabular} 
		\caption{Comparison of the performance between exact and fast greedy approaches, using $\delta = 0.9$ and $\eta =0.01$ both for different $m$. Number of samples $\#K$, stability constant $\Vert \bm{G}_{{\bm{x}^n_{K}}} - \bm{I} \Vert_2$ and CPU time $t$.}   
		\label{table:greedy_time_gaussian}                                                                                                                                                     
	\end{table}  
	
	Indeed, looking at the CPU times in Table \ref{table:greedy_time_uniform} and Table \ref{table:greedy_time_gaussian}, we observe that with the fast methods, the computational times are divided by about $m\log(m)$.\\
	%\bibliographystyle{plain}
	%\bibliography{biblio_article1}

\end{document}